\documentclass[reqno]{amsart}
\usepackage{amssymb,amscd,bbm,graphicx,mathtools,verbatim}
\input MnSymbol.tex
\usepackage[colorlinks]{hyperref}

\newtheorem{thm}{Theorem}[section]

\newtheorem{lem}[thm]{Lemma}
\newtheorem{prop}[thm]{Proposition}

\theoremstyle{definition}
\newtheorem{dfn}[thm]{Definition}

\newtheorem{hyp}[thm]{Hypothesis}

\numberwithin{equation}{section}

\DeclarePairedDelimiter{\norm}{\lVert}{\rVert}
\newcommand{\lVVert}{\lvert\mkern-1.5mu\lvert\mkern-1.5mu\vert}
\newcommand{\rVVert}{\rvert\mkern-1.5mu\rvert\mkern-1.5mu\vert}
\DeclarePairedDelimiter{\tnorm}{\lVVert}{\rVVert}
\DeclarePairedDelimiter{\scp}{\langle}{\rangle}
\DeclarePairedDelimiter{\dscp}{\llangle}{\rrangle}

\renewcommand{\Im}{\operatorname{Im}}
\newcommand{\supp}{\operatorname{supp}}

\newcommand{\iu}{\mathrm{i}}

\newcommand{\dom}{\operatorname{dom}}
\newcommand{\ran}{\operatorname{ran}}

\newcommand{\BVl}{\operatorname{BV}_{\rm loc}}

\newcommand{\sm}[1]{\big(\begin{smallmatrix}#1\end{smallmatrix}\big)}
\newcommand{\bsm}[1]{\begin{pmatrix}#1\end{pmatrix}}
\newcommand{\bb}[1]{{\mathbb{#1}}}
\newcommand{\mc}[1]{{\mathcal{#1}}}
\newcommand{\id}{\mathbbm 1}
\renewcommand{\vec}[1]{\overset{{}_{\rightharpoonup}}{#1}}
\newcommand{\cev}[1]{\overset{{}_{\leftharpoonup}}{#1}}
\newcommand{\cevec}[1]{\overset{{}_{\leftrightarrow}}{#1}}
\newcommand{\setsum}{\sqcup}
\newcommand{\dirsum}{\uplus}
\newcommand{\ov}{\overline}
\newcommand{\ol}{{\overline{\lambda}}}
\newcommand{\la}{\lambda}
\newcommand{\loc}{{\rm loc}}
\newcommand{\mx}{{\rm max}}
\newcommand{\mn}{{\rm min}}
\newcommand{\mnp}[1]{\marginpar{\textsf{#1}}}
\newcommand{\sas}{\;\;\text{and}\;\;}

\newcommand{\<}{\langle}
\renewcommand{\>}{\rangle}

\begin{document}
\title[Self-adjoint extensions]{Self-adjoint extensions of symmetric relations associated with systems of ordinary differential equations with distributional coefficients}
\author{Steven Redolfi and Rudi Weikard}
\address{Department of Mathematics, University of Alabama at Birmingham, Birmingham, AL 35226-1170, USA}
\email{stevenre@uab.edu, weikard@uab.edu}
\date{\today}

\dedicatory{Dedicated to Bernard Helffer on the occasion of his 75th birthday}



\begin{abstract}
We study the extension theory for the two-dimensional first-order system $Ju' +qu = wf$
of differential equations on the real interval $(a,b)$ where $J$ is a constant,
invertible, skew-hermitian matrix and $q$ and $w$ are matrices whose entries are real
distributions of order $0$ with $q$ hermitian and $w$ non-negative. Specifically,
we characterize the boundary conditions for solutions $u$ in the closure of the minimal
relation, as well as describe the properties of quasi-boundary conditions which yield
self-adjoint extensions. We then apply these ideas to a popular extension of non-negative
minimal relations: the Krein-von Neumann extension. For more context on how the Krein-von Neumann
is defined, an appendix shows a construction of the Friedrichs extension
from which the Krein-von Neumann is traditionally defined.
\end{abstract}
\maketitle

\section{Introduction}
In this paper we investigate the extension theory of symmetric linear relations associated with the differential equation
\begin{equation}\label{de}
Ju'+(q-\la w)u=wf
\end{equation}
on the interval $(a,b)\subset\bb R$ where $\la$ is a complex parameter and the coefficients $J$, $q$, and $w$ satisfy the following hypothesis.
\begin{hyp}\label{hyp:m}
$J=\sm{0&-1\\ 1&0}$ and $q$ and $w$ are $2\times 2$-matrices whose entries are real distributions of order $0$ on the interval $(a,b)\subset\bb R$; $q$ is hermitian and $w$ is non-negative.
\end{hyp}

We assume that $f$ is in $L^2(w)$, a Hilbert space to be defined below, and seek solutions among the functions of locally bounded variation.
In this case $u'$ is a distribution of order $0$ and one may define the products $qu$, $wu$, and $wf$ also as distributions of order $0$.
Therefore it makes sense to pose equation \eqref{de}.
We emphasize that one cannot allow the coefficients to be rougher than what we have specified above.
If they were rougher, one would have to admit solutions that are rougher than functions of locally bounded variation and it would then be impossible to define the products $qu$ and $wu$.
To be sure, our approach is general enough to include the classical Sturm-Liouville equation, the analogous difference equations, as well as equations of Dirac-type.
Let us also highlight that, under the given circumstances, the existence and uniqueness of solutions to initial value problems for equation \eqref{de} cannot be guaranteed.

In \cite{MR4047968} the basic spectral theory for equation \eqref{de} was developed (even when the size of the system is more than $2$) but it was assumed
that, when $\la=0$, initial value problems for equation \eqref{de} have unique solutions.
Associated with the equation \eqref{de} are a maximal linear relation $T_\mx$ and a minimal relation $T_\mn$ satisfying, by construction, $T_\mn\subset T_\mx$.
These satisfy the relationship $T_\mn^*=T_\mx$, the cornerstone of spectral theory showing that $T_\mn$ is a symmetric relation.
Moreover, the self-adjoint restrictions of $T_\mx$ were characterized and a Fourier expansion theorem was shown to hold (under additional, perhaps technical assumptions).

The assumption that for $\la=0$ solutions can be continued uniquely across all of $(a,b)$ was found to be unnecessary for several of the results mentioned above: the relationship $T_\mn^*=T_\mx$ was shown to hold in general in \cite{MR4298818}; in \cite{MR4431055} it was shown that the deficiency indices are finite even though the equation $Ju'+qu=\la wu$ may have infinitely many linearly independent solutions; and in \cite{REDOLFI2024110370} the hypotheses for the existence of a Fourier expansion were greatly weakened.
Other developments were carried out in \cite{MR4410825}, \cite{MR4612314}, and \cite{NGUYEN2024128137} (oscillation theory); in \cite{MR4651276} (Floquet theory); and in \cite{BW} (on Weyl's limit-point/limit-circle classification in the case of two-dimensional systems).

The first to allow coefficients more general than locally integrable ones was (as far as we know) Krein \cite{MR0054078} in 1952 when he modeled a vibrating string.
Further contributions\footnote{This list is, of course, not exhaustive, please see \cite{MR3046408} and \cite{MR4047968} for more details.} were made by Kac \cite{MR0080835}, Feller \cite{MR0068082}, Atkinson \cite{MR0176141}, Mingarelli \cite{MR706255}, Gesztesy and Holden \cite{MR914699}, Savchuk and Shkalikov \cite{MR1756602}, and Eckhardt et al. \cite{MR3046408}.

After recalling and extending some basic facts about the differential equation \eqref{de} in Section \ref{S:2} we discuss in Section \ref{S:3} Weyl's alternative, which establishes a link between the deficiency indices of the minimal relation and the behavior of the coefficients $q$ and $w$ near the endpoints $a$ and $b$.
Section \ref{S:4} is devoted to a formulation of the boundary conditions in more classical terms compared to the abstract formulation in \cite{MR4047968}.
Finally, we discuss the boundary conditions satisfied by the Krein-von Neumann extension of $T_\mn$ (when this is a non-negative relation) in Section \ref{S:5}.
Our definition of the Krein-von Neumann extension is via its relationship to the Friedrichs extension.
Since its definition and main properties are not so well-known in the context of linear relations we present our point of view on the matter in Appendix \ref{A:A}.

We close this introduction with a few words on notations.
The transpose of a matrix or a vector $x$ is written as $x^\top$.
In particular, we may write vectors in $\bb C^n$ as transposes of rows for typesetting purposes.
By $[u]$ we denote the class of all functions which are equal to $u$ almost everywhere with respect to the measure $w$.
The symbol $\id$ stands for the identity operator (in various contexts) and the closure of a set $E$ is written as $\ov E$.
Lebesgue measure and the Dirac measure concentrated on the point $k$ are denoted by $dx$ and $\delta_k$, respectively.

We say that $u$ has finite norm (or zero norm) near $b$ if there is a $c\in(a,b)$ such that $\int_{(c,b)}u^*wu<\infty$ (or $\int_{(c,b)}u^*wu=0$).

We have to distinguish between four kinds of sums of sets.
Given subspaces $A$ and $B$ of some vector space $V$ we use $A\setsum B$ to denote the span of $A\cup B$ and,  when $A\cap B=\{0\}$, we use instead $A\uplus B$ to emphasize that the sum is direct.
If $V$ is an inner product space and $A$ and $B$ are orthogonal to each other we write for their direct sum $A \oplus B$.
Finally, if $S$ and $T$ are relations, then $S+T=\{(u,s+t):(u,s)\in S\text{ and }(u,t)\in T\}$.

\section{Preliminaries}\label{S:2}
\subsection{The differential equation}
Given a distribution $h$ of order $0$, Riesz's representation theorem provides a unique positive Borel measure $\mu$ and a unique function $g\in
L^1_\loc(\mu)$ such that $h(\psi)=\int \psi g\mu$ for every $\psi\in C^\infty_0((a,b))$.
Conversely, for any such measure $\mu$ and function $g$ the assignment $\psi\mapsto \int \psi g\mu$ is a distribution of order $0$.
Thus we have a one-to-one correspondence between distributions of order~$0$ and local\footnote{$g\mu$ is a complex measure on each compact subset of $(a,b)$.} measures $g\mu$.
It is standard to identify $h$ and $g\mu$ and, in such a setting, to write $h(E)=\int_Eg\mu$, whenever this integral exists.
Any cumulative distribution function $H$ of $h=g\mu$ is a function of locally bounded variation which itself gives rise to a distribution of order $0$ (also denoted by $H$) by setting $H(\psi)=\int H\psi\; dx$.
The distributional derivative of $H$ is $H'=h$ and the (local) Lebesgue-Stieltjes measure generated by $H$ is $dH=g\mu$.

In order to make sense of the differential equation we require that $u$ is a function of locally bounded variation%
\footnote{In fact we ask that $u$ is \emph{balanced}, a term defined below.}
so that $u'$ is a distribution of order $0$.
The product $qu$ is then defined by $(qu)(\psi)=\int \psi u\,dQ$ where $Q$ is a cumulative distribution function (or anti-derivative) of $q$.
To define $wf$ we note first that $w$ is a positive (matrix-valued) measure.
Thus, if $f\in L^1_\loc(w)$, we may define $(wf)(\psi)=\int \psi f\,dW$ where $W$ is a cumulative distribution function (or anti-derivative) of $w$.
These considerations\footnote{For more details the reader may consult \cite{MR4047968}.} show that each term in $Ju'+qu=wf$ is a distribution of order $0$ so that it makes sense to pose the equation.
The same considerations work, of course, for the equation $Ju'+(q-\la w)u=wf$ when $\la$ is a complex parameter.
We emphasize that, as the coefficients in the equation $Ju'+qu=wf$ become rougher, the solutions $u$ must, too.
At some point it becomes impossible to define the products $qu$ and $wu$.
This point is reached when $q$ or $w$ are rougher than distributions of order $0$.

If the measures $q$ or $w$ have atomic parts existence and uniqueness of solutions of initial value problems are in question.
To investigate this issue we need to introduce some notation.
Let $g$ be a function of locally bounded variation.
Then we denote its right-continuous and left-continuous versions respectively by $g^+$ and $g^-$,
i.e., we set $g^+(s)=\lim_{t\downarrow s}g(t)$ and $g^-(s)=\lim_{t\uparrow s}g(s)$.
We also introduce $g^\#=(g^++g^-)/2$ and call $g$ \emph{balanced} if $g=g^\#$.
We denote the space of ordered pairs of balanced functions of locally bounded variation on $(a,b)$ by $\BVl^\#((a,b))^2$.
Then we define the matrices $\Delta_q(x)=q(\{x\})=Q^+(x)-Q^-(x)$, $\Delta_w(x)=w(\{x\})=W^+(x)-W^-(x)$, and
$$B_\pm(x,\la)=J\pm\frac12(\Delta_q(x)-\lambda\Delta_w(x)).$$

We gather the most basic facts about the matrices $B_\pm(x,\la)$.
\begin{prop}\label{P:2.1}
The following statements hold true:
\begin{enumerate}
  \item Each of $B_\pm(x,\la)$ and $B_\pm(x,\ol)$ is invertible, if any one of them is.
  \item $\ker B_-(x,\la)\cap\ker B_+(x,\la)$ is always trivial.
  \item Neither $B_+(x,\la)$ nor $B_-(x,\la)$ can be $0$.
  \item $\ran B_-(x,\la)\cap\ran B_+(x,\la)$ is either trivial or all of $\bb C^2$.
\end{enumerate}
\end{prop}

\begin{proof}
Since $B_+(x,\la)=-B_-(x,\la)^\top=-B_-(x,\ol)^*$ we obtain the first claim.
To prove the second, assume that $z\in\ker B_-(x,\la)\cap\ker B_+(x,\la)$.
Then
$$0=B_-(x,\la)z+B_+(x,\la)z=2J z$$ which implies $z=0$.
For the proof of (3) assume, by way of contradiction, that $B_+(x,\la)=0$.
From (2) we get then that $B_-(x,\la)$ is invertible.
This is, in view of (1), impossible.
Finally, if $B_\pm(x,\la)$ are invertible, then the intersection of their ranges is all of $\bb C^2$.
Otherwise, $\dim \ran B_\pm(x,\la)=\dim\ker B_\mp(x,\ol)=1$ and
\[\ran B_-(x,\la)\cap\ran B_+(x,\la)=(\ker B_+(x,\ol))^\perp\cap(\ker B_-(x,\ol))^\perp=\{0\}.\qedhere\]
\end{proof}

With these preparations we can now state the existence and uniqueness theorem whose proof may be found in \cite{MR4047968}.

\begin{thm}\label{t:ex-un}
Suppose that Hypothesis \ref{hyp:m} holds.
Moreover, assume $f\in L^1_\loc(w)$, $\lambda\in\mathbb{C}$, $(c,d)\subset(a,b)$, $x_0\in(c,d)$, and $u_0\in\mathbb{C}^2$.
If the matrices $B_\pm(x,\lambda)$ are invertible for all $x\in(c,d)$, then the initial value problem $Ju'+(q-\lambda w)u=wf$, $u(x_0)=u_0$ has a unique balanced solution $u$ in $(c,d)$.

As long as $q$ and $w$ are finite measures near $c$ (or $d$), we may pose the initial condition at $c$ (or $d$), even if $c$ (or $d$) is infinite.
\end{thm}

If $q$ and $w$ are finite measures near one of the endpoints $a$ or $b$, that endpoint is called \emph{regular} and otherwise \emph{singular}.
Note that an endpoint may be regular even if it is infinite.

The reason we are asking for balanced solutions is the integration by parts formula which states that
\[\int_{(c,d)} (udv+vdu)=(uv)^-(d)-(uv)^+(c)+(2\kappa-1)\int_{(c,d)}(v^+-v^-)du\]
whenever $u$ and $v$ are functions of bounded variation on the interval $(c,d)$ satisfying $u=\kappa u^++(1-\kappa)u^-$ and $v=\kappa v^++(1-\kappa)v^-$ for some fixed parameter $\kappa$.
Thus, when $\kappa=1/2$, the integration by parts formula takes its classical form.

For the differential equation to hold at the singleton $x$ it is necessary and sufficient that $Ju'(\{x\})+(q(\{x\})-\la w(\{x\}))u(x)=w(\{x\})f(x)$, or, equivalently, that
$$B_+(x,\la)u^+(x)-B_-(x,\la)u^-(x)=\Delta_w(x)f(x).$$
Thus, if $B_+(x,\la)$ is not invertible, we could be in one of the following two situations: (i) it may be impossible to continue a solution given on $(a,x)$ beyond $x$ or (ii) there are infinitely many ways to continue a solution on $(a,x)$ beyond $x$.
An analogous statement holds, of course, if $B_-(x,\la)$ is not invertible.

Since the entries of $Q$ and $W$ are of locally bounded variation, they can have at most countably many (jump) discontinuities.
Points $x$ where both $Q$ and $W$ are continuous, i.e., where $q(\{x\})=w(\{x\})=0$, will be called \emph{points of continuity}.
There are at most countably many points $x$ in $(a,b)$ where $B_\pm(x,\la)$ fail to be invertible.
Since $\det B_-(x,\cdot)$ is a polynomial of degree at most $2$, it has at most two zeros unless it is identically equal to zero.
Setting $\Lambda_x=\{\la\in\bb C: \det B_\pm(x,\la)=0\}$ and $\Lambda=\bigcup_{x\in(a,b)}\Lambda_x$ it follows that $\Lambda$ is either equal to $\bb C$ or else countable.
On any subinterval of $(a,b)$ on which $q$ gives rise to a finite measure we find that $\sum_{k=1}^\infty \norm{\Delta_q(x_k)}$ must be finite, when $k\mapsto x_k$ is a sequence of distinct points in that interval.
It follows that $\Xi_0=\{x\in(a,b): \det B_\pm(x,0)=0\}$ is a discrete set.
One shows similarly that, for any fixed complex number $\la$ the set $\Xi_\la=\{x\in(a,b): \det B_\pm(x,\la)=0\}$ is discrete.
In particular, points in $\Xi_\la$ cannot accumulate at a regular endpoint.
We emphasize, however, that $\Xi=\bigcup_{\la\in\bb C}\Xi_\la$ need not be discrete.

If $B_\pm(x,\lambda)$ are invertible for all $x\in(c,d)$, Theorem \ref{t:ex-un} guarantees, in particular, the existence of a fundamental matrix $U(\cdot,\la)$ for $Ju'+(q-\la w)u=0$ in $(c,d)$.
Using Lemma 3.2 of \cite{MR4047968} we obtain that $U^+(x,\ol)^*JU^+(x,\la)=U^-(x,\ol)^*JU^-(x,\la)$ does not depend on $x\in(c,d)$.
In particular, if $U^+(x_0,\la)=U^+(x_0,\ol)$ is a symplectic\footnote{Recall that a matrix $S\in\bb C^{2\times2}$ is called symplectic, if $S^*JS=J$.
A matrix $S$ is symplectic if and only if $S=\alpha R$ where $\alpha$ is a complex number of modulus $1$ and $R$ is a real matrix with determinant $1$.}
matrix for some point $x_0\in(c,d)$ we obtain the Wronskian relation
\begin{equation}\label{wronski}
U^\pm(x,\ol)^*JU^\pm(x,\la)=J.
\end{equation}

Next we state the variation of constants formula.
\begin{thm}\label{t:vc}
Let $(c,d)\subset(a,b)$ be an interval such that $B_\pm(x,\lambda)$ are invertible for all $x\in(c,d)$, let $x_0$ be a point of continuity in $(c,d)$, and let $u_0\in\bb C^2$.
Suppose $U(\cdot,\la)$ is a fundamental matrix on $(c,d)$ such that $S=U(x_0,\la)=U(x_0,\ol)$ is a symplectic matrix.
If $f\in L^1_\loc(w)$ and the function $u=u^\#$ is given by
\begin{equation}\label{vc1}
u^-(x)=U^-(x,\la)\Bigl(u_0+J^{-1} \int_{[x_0,x)}U(\cdot,\ol)^*wf\Bigr)\quad \text{for $x\geq x_0$}
\end{equation}
and
\begin{equation}\label{vc2}
u^+(x)=U^+(x,\la)\Bigl(u_0-J^{-1} \int_{(x,x_0]}U(\cdot,\ol)^*wf\Bigr)\quad \text{for $x\leq x_0$,}
\end{equation}
then the pair $(u,f)$ satisfies $Ju'+(q-\la w)u=wf$ on $(c,d)$ and $u(x_0)=Su_0$.

Conversely, if $u$ is a balanced solution of $Ju'+(q-\la w)u=wf$ on $(c,d)$ and $u(x_0)=Su_0$, then $u^\pm$ satisfy \eqref{vc1} and \eqref{vc2}.
\end{thm}

\begin{proof}
When $S=U(x_0,\cdot)=\id$, this is proved in Lemma 3.3 of \cite{MR4047968}.
The more general case follows by using that $U(\cdot,\la)S^{-1}$ is a fundamental matrix satisfying the hypothesis of that Lemma and noticing that $SJ^{-1}S^*=J^{-1}$.
\end{proof}

\subsection{Linear relations}\label{S:2.2}
We define the set $\mc L^2(w)$ as the set of all $\bb C^2$-valued functions $f$ on $(a,b)$ such that $\int f^*wf$ makes sense and is finite.
We emphasize that $\mc L^2(w)$ is a complete semi-inner product space with semi-inner product $\scp{f,g}=\int f^*wg$ but, in general, not a Hilbert space.
The corresponding Hilbert space of classes of functions which are equal almost everywhere with respect to $w$ will be denoted by $L^2(w)$.
Observe that $L^2(w)$ is a subspace of $L^1_\loc(w)$.

We may now define the linear relations
$$\mc T_\mx=\{(u,f)\in \mc L^2(w)\times \mc L^2(w): u\in\BVl^\#((a,b))^2, Ju'+qu=wf\}$$
and
$$\mc T_\mn=\{(u,f)\in \mc T_\mx: \text{$\supp u$ is compact in $(a,b)$}\}.$$

Given two representatives $f$ and $\tilde f$ of some element in $L^2(w)$ we have $wf=w\tilde f$.
Therefore there is no need to distinguish between representatives of the same class as far as the right-hand side of the equation is concerned.
However, two balanced representatives $u$ and $\tilde u$ of $[u]$, both of locally bounded variation, may not satisfy $Ju'+qu=J\tilde u'+q\tilde u$ and this is the reason for our emphasis on the distinction between $\mc L^2(w)$ and $L^2(w)$.
For instance, if $q=0$, $w=\id \delta_0$, and $u(x)=-u(-x)=(\alpha,\beta)^\top$ for $x>0$, then $[u]=0$ but $Ju'+qu=(-2\beta,2\alpha)^\top\neq0$, in general.

Restricting our attention to the Hilbert space $L^2(w)\times L^2(w)$ we define now the relation
$$T_\mx=\{([u],[f])\in L^2(w)\times L^2(w):(u,f)\in\mc T_\mx\}$$
which is called the maximal relation associated with $Ju'+qu=wf$ and the minimal relation
$$T_\mn=\{([u],[f])\in L^2(w)\times L^2(w):(u,f)\in\mc T_\mn\}.$$
Here (and elsewhere) we choose brevity over precision: whenever we have a pair $([u],[f])$ in $T_\mx$ or $T_\mn$ we assume that $u$ is a balanced function of locally bounded variation.

Recall that the adjoint $S^*$ of a relation $S\subset L^2(w)\times L^2(w)$ is defined by
$$S^*=\{(v,g)\in L^2(w)\times L^2(w): \forall (u,f)\in S: \scp{g,u}=\scp{v,f}\}.$$
It was shown in \cite{MR4298818} that $T_\mn^*=T_\mx$ which is, of course, a cornerstone of spectral theory.
In particular, it shows that $T_\mn$ is a symmetric relation.
Let $\mc D_{\la}=\{(u,\la u)\in\mc T_\mx\}$ and $D_{\la}=\{(u,\la u)\in T_\mx\}$.
Then von Neumann's relation
\begin{equation}\label{vNeumann}
T_\mx=\ov{T_\mn}\oplus D_\iu\oplus D_{-\iu}
\end{equation}
holds true.
When $\la$ is not real, the spaces $D_{\la}$ are called deficiency spaces.
Their dimensions are independent of $\la$ as long as $\la$ varies in either the upper or else the lower half plane.
One therefore defines the deficiency indices $n_\pm=\dim D_{\pm\iu}$.

We will make frequent use of Lagrange's identity and therefore provide the following reminder of its substance.
If $(v,g)$ and $(u,f)$ are in $\mc T_\mx$, then $v^*wf-g^*wu=v^*Ju'+v^{\prime*}Ju$ is a finite measure on $(a,b)$.
Integration by parts gives that
\begin{equation}\label{pLagrange}
(v^*Ju)^-(d)-(v^*Ju)^+(c)=\int_{(c, d)}(v^*Ju'+v^{\prime*}Ju)=\int_{(c, d)}(v^*wf-g^*wu)
\end{equation}
whenever $[c,d]\subset(a,b)$.
This implies that the limits $(v^*Ju)^-(b)$ and $(v^*Ju)^+(a)$ exist independently of each other.
We emphasize that these limits may change when $(v,g)$ or $(u,f)$ are replaced by different balanced representatives in their classes.
However their difference does not, since
\begin{equation}\label{Lagrange}
(v^*Ju)^-(b)-(v^*Ju)^+(a)=\<v,f\>-\<g,u\>
\end{equation}
where the right-hand side depends only on the respective classes but not the representatives.
Either of the equations \eqref{pLagrange} and \eqref{Lagrange} is called Green's formula or Lagrange's identity.

Since $\ov u$ satisfies $J\ov u'+q\ov u=\ol w\ov u$ if and only if $u$ satisfies $Ju'+qu=\la wu$, the deficiency indices $n_+$ and $n_-$ coincide and the existence of self-adjoint restrictions of $T_\mx$ is therefore guaranteed.
In \cite{MR4431055} it was shown that the deficiency indices $n_\pm$ do not exceed $2$ and that the self-adjoint restrictions of $T_\mx$ are given by boundary conditions.
Specifically, if $(v_1,g_1), ..., (v_{n_\pm},g_{n_\pm})$ are linearly independent elements of $D_\iu\oplus D_{-\iu}$, define the $n_\pm\times n_\pm$-matrix $G$ whose entries are
$$G_{\ell,k}=\<(g_\ell,-v_\ell),(v_k,g_k)\>=(g_\ell^*J g_k)^-(b)-(g_\ell^*J g_k)^+(a).$$
If $G=0$, then
\begin{equation}\label{bc}
T=\{(u,f)\in T_\mx: (g_j^*Ju)^-(b)-(g_j^*Ju)^+(a)=0, j=1,...,n_\pm\}
\end{equation}
is a self-adjoint restriction of $T_\mx$.
Conversely, given a self-adjoint restriction $T$ of $T_\mx$, there are $n_\pm$ linearly independent elements $(v_j,g_j)$ of $D_\iu\oplus D_{-\iu}$, such that the corresponding matrix $G$ equals $0$ and $T$ is given as in equation \eqref{bc}.
Since $(v,g)$ in $D_\iu\oplus D_{-\iu}$ implies that $(g,-v)$ is also in $D_\iu\oplus D_{-\iu}$, the condition $G=0$ shows that $(g_j,-v_j)\in T$ for $j=1,..., n_\pm$, i.e.,
$$T=\ov T_\mn\oplus\operatorname{span}\{(g_1,-v_1), ..., (g_{n_\pm},-v_{n_\pm})\}.$$

Since $\<(v_j,g_j),(v_j,g_j)\>\neq0$ it is impossible that both $(g_j^*Ju)^+(a)$ and $(g_j^*Ju)^-(b)$ vanish for all $(u,f)\in \mc T_\mx$.
We have therefore three mutually exclusive cases: (1) $(g_j^*Ju)^+(a)=0$ for all $(u,f)\in \mc T_\mx$, (2) $(g_j^*Ju)^-(b)=0$ for all $(u,f)\in \mc T_\mx$, and (3) for some $(u,f)\in \mc T_\mx$ both $(g_j^*Ju)^+(a)$ and $(g_j^*Ju)^-(b)$ are different from zero.%
\footnote{This classification is for the function $g_j$, not its class. If $u_0\in\mc L_0$, then $g_j$ and $g_j+u_0$ describe the same boundary. The former could be in case (1) while the latter could be in case (3).}
Accordingly there are three kinds of boundary conditions.
We have separated boundary conditions, if we have case (1) or (2) for every $j\in\{1,...,n_\pm\}$; we have coupled boundary conditions, if we have case (3) for every $j\in\{1,...,n_\pm\}$; and we have mixed boundary conditions otherwise.
We will see later that in the present circumstances mixed boundary conditions do not occur.

In the main part of the paper we will occasionally make use of the following observation.
If $\Lambda$ is not the full plane it is a countable set and therefore there is a real number $\mu$ which is not in $\Lambda$.
We define $\tilde q=q-\mu w$ and $\tilde w=w$ and the corresponding relations $\tilde T_\mx$, etc.
Clearly we have $\tilde{\mc T}_\mx=\mc T_\mx-\mu\id$, i.e.,
$$(u,f)\in \mc T_\mx \Leftrightarrow (u,f-\mu u)\in \tilde{\mc T}_\mx.$$
Hence analogous relations hold for $\mc T_\mn$, $T_\mx$ and $T_\mn$.
But how about $\ov{T_\mn}$?
Suppose $(u,f)\in\ov{T_\mn}=T_\mx^*$ and $(v,g)\in \tilde{T}_\mx$.
Then $(v,g+\mu v)\in T_\mx$ and, since $\mu\in\bb R$,
$$\scp{v,f-\mu u}-\scp{g,u}=\scp{v,f}-\scp{g+\mu v,u}=0.$$
But this means that $(u,f-\mu u)\in \tilde T_\mx^*=\ov{\tilde T_\mn}$.
We may show similarly the other direction so that we have
\begin{equation}\label{20230425.1}
(u,f)\in \ov{T_\mn} \Leftrightarrow (u,f-\mu u)\in \ov{\tilde T_\mn}.
\end{equation}

\subsection{Definiteness and the lack thereof}\label{mcl0}
Define the space $\mc L_0$ of all solutions of $Ju'+qu=0$ with norm $0$, i.e.,
$$\mc L_0=\{u\in\BVl^\#((a,b))^2: Ju'+qu=0, \norm u=0\}.$$
Note that $\norm u=0$ if and only if $wu=0$.
Therefore, if $u\in\mc L_0$, we get $Ju'+qu=\la wu$ for all $\la\in\bb C$.
This implies that no non-trivial element of $\mc L_0$ can be compactly supported unless $\Lambda=\bb C$.

If $u_0\in\mc L_0$ and $(u,f)\in \mc T_\mx$, then $(u+u_0,f)$ is also in $\mc T_\mx$ and $([u],[f])$ and $([u+u_0],[f])$ are the same in $T_\mx$.
Conversely, if $(u,f)$ and $(\tilde u,\tilde f)$ are both representatives of $([u],[f])$, then $u-\tilde u\in\mc L_0$.
If $\mc L_0$ is trivial we call the problem posed by $Ju'+qu=wf$ \emph{definite} and otherwise \emph{non-definite}.

In the definite case the class $[u]$ in the pair $([u],[f])\in T_\mx$ has a unique representative in $\BVl^\#((a,b))^2$ satisfying $Ju'+qu=wf$.
In that case we can identify $T_\mx$ with $\mc T_\mx$, $T_\mn$ with $\mc T_\mn$, etc.

If $\Lambda=\bb C$, an example given in Section 7 of \cite{MR4431055}, shows that $\mc L_0$ may be infinite-dimensional.
However, if $\Lambda\neq\bb C$ and $\la\in\bb C\setminus\Lambda$, the space of all solutions of $Ju'+qu=\la wu$ is two-dimensional and $\mc L_0$ is at most two-dimensional.
If $\Lambda\neq\bb C$ and $\dim\mc L_0=2$, we must have $w=0$, a case which we may ignore since $L^2(w)$ is then trivial.
We also have the following observation.
\begin{lem}\label{L:2.4}
If $\Lambda\neq \bb C$ and $\mc L_0\neq\{0\}$, then $\Lambda\subset\bb R$.
\end{lem}

\begin{proof}
By way of contraposition, we assume there are non-real elements $\la$ and $\ol$ of $\Lambda$ and $\Lambda\neq\bb C$.
Then there is an $x\in(a,b)$ such that $\det B_+(x,\cdot)$ has the zeros $\la$ and $\ol$.
This implies that $\Delta_w(x)$ is positive definite.
If $\mu$ is not in $\Lambda$, then $Ju'+qu=\mu wu$ has a two-dimensional space of solutions and any non-trivial solution $u$ of this equation vanishes nowhere.
Hence $\norm u\geq u(x)^*\Delta_w(x)u(x)>0$, i.e., only the trivial solution of $Ju'+qu=\mu wu$ has norm $0$.
Thus $\mc L_0$ is trivial.
\end{proof}

Suppose $\la\in\bb C\setminus\Lambda$, $x_0\in(a,b)$ is a point of continuity, and $U(\cdot,\la)$ is a fundamental matrix for $Ju'+qu=\la wu$ satisfying $U(x_0,\la)=\id$.
Define
$$\mathbf{B}(\la)=\left\{\int U(\cdot,\ol)^*wf: \text{$f\in L^2(w)$, $\supp f$ compact in $(a,b)$}\right\}.$$
Suppose $z\in\bb C^2$ is orthogonal to $\mathbf{B}(\la)$, i.e., $\int f^*w U(\cdot,\ol)z=0$ for all compactly supported functions $f\in L^2(w)$.
Then $w U(\cdot,\ol)z$ is the zero measure so that $\|U(\cdot,\ol)z\|=0$.
Hence $u_0=U(\cdot,\ol)z\in\mc L_0$.
Since $u_0$ is a solution of $Ju'+qu=\la wu_0$ for any $\la\in\bb C$ it follows that $z=u_0(x_0)$ is independent of $\la$.
Thus $\mathbf{B}$ is also independent of $\la\in\bb C\setminus\Lambda$ and we will disregard its argument $\la$ in the following.
We emphasize, however, that $\mathbf{B}$ depends on the choice of the fundamental matrix $U(\cdot,\la)$.
Since $\bb C^2$ is finite-dimensional we may find, at least for a fixed $\la\in\bb C\setminus\Lambda$, an interval $[c,d]$, whose endpoints are points of continuity, such that $\mathbf{B} =\{\int U(\cdot,\ol)^*wh: h\in L^2(w), \supp{h}\subset[c,d]\}$.

\begin{lem}\label{L:8.4mod1}
Suppose $\dim\mc L_0=0$ and $\la\in\bb C\setminus\Lambda$.
If $u$ satisfies $Ju'+qu=w(\la u+f)$ and has finite norm near $a$, then there is an element $(v,\la v+g)\in \mc T_\mx$ such that $v=u$ near $a$ and $v=0$ near $b$.
A similar statement holds with the roles played by $a$ and $b$ reversed.
\end{lem}

\begin{proof}
Since $\dim\mc L_0=0$ we have $\mathbf B=\bb C^2$ and hence there is an $h\in L^2(w)$ with support in $[c,d]$ such that $\int U(\cdot,\ol)^*wh =-U(c,\ol)^*Ju(c)$.
Now let $g=f$ on $(a,c)$ and $g=h$ on $(c,b)$.
Define $v$ to be the unique solution of the initial value problem $Jv'+qv=w(\la v+g)$, $v(c)=u(c)$.
Then $v=u$ on $(a,c)$ and, when $x\geq d$,
$$(U(\cdot,\ol)^*Jv)^-(x)=(U(\cdot,\ol)^*Jv)^+(c)+\int_{(c,d)} U(\cdot,\ol)^*w h=0$$
according to Lagrange's identity \eqref{pLagrange}. Thus $v=0$ on $(d,b)$.
\end{proof}

We now have a closer look at the case when $\Lambda\neq\bb C$ (and hence countable) and $\mc L_0$ is one-dimensional.
It follows that $\mathbf{B}$ is then also one-dimensional.
Since, for real $\mu$ not in $\Lambda$, the entries of the fundamental matrix $U(\cdot,\mu)$ are real-valued, the space $\mathbf{B}$ is spanned by a real unit vector $e\in\bb C^2$.
We also define $e^\perp=Je$ which is perpendicular to $e$.

Suppose $g\in L^2(w)$ and $v$ satisfies $Jv'+qv=w(\la v+g)$.
We may write $v=U(\cdot,\la)(\alpha e+\beta e^\perp)$ for suitable functions $\alpha$ and $\beta$, since the columns of $U(x,\la)$ are linearly independent for every $x\in(a,b)$.
According to equation \eqref{wronski} we have $U(\cdot,\ol)^*Jv=\alpha e^\perp-\beta e$.
If $s,t$ are points of continuity and $s\leq t$ Lagrange's identity \eqref{pLagrange} gives therefore
\begin{multline*}
(\alpha(t)-\alpha(s)) e^\perp-(\beta(t)-\beta(s)) e=U(t,\ol)^*Jv(t)-U(s,\ol)^*Jv(s)\\
 =\int_{(s,t)} U(\cdot,\ol)^*w g.
\end{multline*}
The right-hand side is an element of $\mathbf B$ and hence a multiple of $e$.
This implies that, at points $x$ of continuity, $\alpha$ is constant and that
\begin{equation}\label{20230424.1}
\beta(x)=\begin{cases} \beta(x_0)+e^*\int_{(x,x_0)} U(\cdot,\ol)^*wg&\text{if $x\leq x_0$},\\
                       \beta(x_0)-e^*\int_{(x_0,x)} U(\cdot,\ol)^*wg&\text{if $x\geq x_0$}.\end{cases}
\end{equation}
We remark that $\alpha$ and $\beta$ may not be balanced.
In particular, $\alpha(x)$ may indeed be different from $\alpha(x_0)$ if $\Delta_q(x)$ or $\Delta_w(x)$ does not vanish.

\begin{lem}\label{L:8.4mod2}
Suppose $\dim\mc L_0=1$ and $\la\in\bb C\setminus\Lambda$.
If $u$ satisfies $Ju'+qu=w(\la u+f)$ and has finite norm near $a$, then there is an element $(v,\la v+g)\in \mc T_\mx$ such that $v$ equals $u$ near $a$ and is a constant multiple of $U(\cdot,\la)e$ near $b$.
A similar statement holds with the roles played by $a$ and $b$ reversed.
\end{lem}

\begin{proof}
Suppose $g=f$ on $(a,c)$ and $g=h$ (to be determined presently) on $(c,b)$.
Let $v$ be the solution of the initial value problem $Jv'+qv=w(\la v+g)$, $v(c)=u(c)$.
Then $v=u$ on $(a,c)$.
If $h\in L^2(w)$ has support in $[c,d]$ and $x\geq d$ is a point of continuity, then
$$U(x,\ol)^*Jv(x)=\alpha e^\perp-\beta(c) e +\int_{(c,d)} U(\cdot,\ol)^*w h$$
according to Lagrange's identity \eqref{pLagrange}.
Now choose $h$ so that $e^*\int_{(c,d)} U(\cdot,\ol)^*w h=\beta(c)$.
Then  $U(x,\ol)^*Jv(x)=\alpha e^\perp$.
Since, by equation \eqref{wronski}, $-U(x,\la)J$ is the inverse of $U(x,\ol)^*J$ and $-J^2=\id$, this implies $v(x)=\alpha U(x,\la)e$ on $(d,b)$.
\end{proof}

We prove now some basic facts about $\ov{T_\mn}$.
\begin{lem}\label{L:2.3}
Suppose $\dim\mc L_0=1$ and $\Lambda\neq\bb C$.
If $\mu\in\bb R\setminus\Lambda$, $([u],[\mu u+f])\in\ov{T_\mn}$, and $x$ is a point of continuity, then $u(x)=\beta(x)U(x,\mu)e^\perp$ where $\beta$ is given by \eqref{20230424.1} but with $\la=\mu$ and $g=f$.
Moreover, the limits
\begin{equation}\label{20230416.1}
\lim_{x\uparrow b}e^*\int_{[x_0,x)} U(\cdot,\mu)^*wf \quad\text{and}\quad \lim_{x\downarrow a}e^*\int_{(x,x_0]} U(\cdot,\mu)^*wf
\end{equation}
both exist.
Finally, $\int U(\cdot,\mu)^*wf=0$ unless $e^*JU(x_0,\mu)^*Jv(x_0)=0$ for all $(v,\mu v+g)\in\mc T_\mx$.
\end{lem}

\begin{proof}
When $x$ is a point of continuity we have $u(x)=U(x,\mu)(\alpha e+\beta(x)e^\perp)$ where $\alpha\in\bb C$ and $\beta$ is given by \eqref{20230424.1} but with $\la=\mu$ and $g=f$.

To prove the first claim we will construct an element $(v,\mu v+h)\in\mc T_\mx$ such that $0=(v^*Ju)^-(b)-(v^*Ju)^+(a)=-\alpha$.
There is an $h\in L^2(w)$ with support in $[c,d]$ such that $\int U(\cdot,\mu)^*wh=e$.
Let $v=0$ to the left of $c$ and define it elsewhere with the aid of the variation of constants formula \eqref{vc1}, choosing $u_0=0$, $x_0=c$, and $f=h$.
Then $v$ satisfies $Jv'+qv=w(\mu v+h)$, $v=0$ near $a$, and $v=-U(\cdot,\mu)e^\perp$ on $(d,b)$.
Thus $wv=0$ on $(d,b)$ so that $(v,\mu v+h)\in\mc T_\mx$.
Lagrange's identity \eqref{Lagrange} gives $(v^*Ju)^-(b)=(v^*Ju)^-(b)-(v^*Ju)^+(a)=0$.
But, for $x\in(d,b)$, we have, using \eqref{wronski},
$$(v^*Ju)(x)=e^*(\alpha e+\beta(x)e^\perp)=-\alpha\|e\|^2.$$
Thus $\alpha=0$ as promised.

To prove the second claim consider an arbitrary element $(v,\mu v+g)\in\mc T_\mx$.
Then $v(x)=U(x,\mu)(\gamma e+\delta(x)e^\perp)$ for some number $\gamma$ and some function $\delta$ when $x$ is a point of continuity.
Now we get
\begin{equation}\label{20230425.2}
(v^*Ju)(x)=(\ov\gamma e^*+\ov{\delta(x)}e^{\perp*})U(x,\mu)^*JU(x,\mu)\beta(x)e^\perp=-\ov\gamma \beta(x)
\end{equation}
at points of continuity.
Lagrange's identity \eqref{pLagrange} implies that $(v^*Ju)(x)$ has limits as $x$ tends to $a$ as well as $b$ and then \eqref{20230424.1} shows that the limits in \eqref{20230416.1} exist.

Finally, \eqref{20230425.2} entails that $0=(v^*Ju)^-(b)-(v^*Ju)^+(a)=\ov\gamma \int_{(a,b)} U(\cdot,\mu)^*wf$ which proves our last claim since $\gamma=-e^*JU(x_0,\mu)^*Jv(x_0)$.
\end{proof}

\subsection{Really bad points}
We call a point $x\in(a,b)$ a really bad point if the matrices $B_\pm(x,\la)$ fail to be invertible for all $\la\in\bb C$.
The presence of such points requires special considerations in many of the proofs to come.
Here we investigate their most basic properties.

If $x$ is a really bad point, then the ranks of $B_+(x,\la)$ and $B_-(x,\la)$ are equal to $1$ for all $\la$ in $\bb C$.
In particular, there are then vectors $\omega_\pm\in\bb R^2$ such that $B_+(x,0)=\omega_+\omega_-^*$ and $B_-(x,0)=-\omega_-\omega_+^*$.
The vectors $\omega_\pm$ are linearly independent since $\omega_-=c \omega_+$ for some $c\in\bb R$ would imply that $2J=B_+(x,0)+B_-(x,0)=0$.

We can now characterize the really bad points.
\begin{lem}\label{l:2.5}
The following are equivalent:
\begin{enumerate}
  \item $\Lambda_x=\bb C$.
  \item $0\in\Lambda$ and $\ran\Delta_w(x)$ is contained in $\ran B_+(x,0)$ or in $\ran B_-(x,0)$.
  \item $\ran B_+(x,\la)\cap\ran B_-(x,\la)=\{0\}$ for all $\la\in\bb C$.
\end{enumerate}
\end{lem}

\begin{proof}
First we show that (1) implies (2).
Since $\Lambda_x=\bb C$ the determinants of $B_\pm(x,\la)$ cannot both be polynomials of degree $2$ implying that $\Delta_w(x)$ is not invertible.
Therefore we may suppose that $\Delta_w(x)=\omega\omega^*$ for some $\omega\in\bb C^2$.
If $\omega=0$ we are done so we assume now $\omega\neq0$.
Let $\la$ be any non-zero complex number.
Then we know that $B_\pm(x,\la)$ have rank $1$.
Hence $B_+(x,\la)=B_+(x,0)-\frac\la2\Delta_w(x)$ has a non-trivial kernel spanned by a vector $v$.
It follows that $2\omega_+(\omega_-^*v)=\la\omega(\omega^*v)$.
If $\omega^*v=0$ we also have $\omega_-^*v=0$ proving that $\omega_-$ and $\omega$ are parallel, i.e., $\ran\Delta_w(x)=\ran B_-(x,0)$.
If $\omega^*v\neq0$ we get instead that $\omega_+$ and $\omega$ are parallel, i.e., $\ran\Delta_w(x)=\ran B_+(x,0)$.
This concludes the proof of the first step.

Now we prove (2) implies (3).
We assume that $\ran\Delta_w(x)\subset\ran B_+(x,0)=\operatorname{span}\{\omega_+\}$, the other case being treated in a similar fashion.
Our assumption is $\omega=c\omega_+$ for some complex number $c$ (possibly $0$).
Then we have $B_-(x,\la)=-\omega_-\omega_+^*+\frac\la2\omega\omega^*=(\frac12\la |c|^2\omega_+-\omega_-)\omega_+^*$.
Therefore its range does not coincide with the range of $B_+(x,\la)$.

That (3) implies (1) is clear.
\end{proof}

Suppose $(a,b)$, $q$, and $w$ are given with associated relations $\mc T_\mx$ and $T_\mx$.
Moreover, let $x_1\in(a,b)$ be such that $\ran\Delta_w(x_1)\subset \ran B_+(x_1,0)\neq\bb C^2$, so that $x_1$ is a really bad point.
If $(u,f)\in\mc T_\mx$ define $(u_0,f_0)$ by setting $u_0=(u\chi_{(a,x_1)})^\#$ and $f_0=f\chi_{(a,x_1)}$.
Then $(u_0,f_0)\in\mc T_\mx$.
Similarly, if $u_1=(u\chi_{(x_1,b)})^\#$ and $f_1=f\chi_{[x_1,b)}$, then $(u_1,f_1)\in\mc T_\mx$.
Moreover, $([u_0],[f_0])$ and $([u_1],[f_1])$ are orthogonal elements of $T_\mx$.
To see this note that
$$\scp{([u_1],[f_1]),([u_0],[f_0])}=u_1(x_1)^*\Delta_w(x_1) u_0(x_1)=\frac14 u^+(x_1)^*\Delta_w(x_1) u^-(x_1)=0$$
where the last equality follows since $u^-(x_1)\in\ker B_-(x_1,0)=(\ran B_+(x_1,0))^\perp$ so that $\Delta_w(x_1) u^-(x_1)=0$.
Thus, defining
$$T_0=\{([(u\chi_{(a,x_1)})^\#],[f\chi_{(a,x_1)}]):(u,f)\in\mc T_\mx\}$$
and
$$T_1=\{([(u\chi_{(x_1,b)})^\#],[f\chi_{[x_1,b)}]):(u,f)\in\mc T_\mx\}$$
we obtain $T_\mx=T_0\oplus T_1$.
Note that $\ker B_-(x_1,0)\subset \ker\Delta_w(x_1)$.
Thus, setting $v=(u\chi_{(a,x_1)})^\#$ when $u\in\dom\mc T_\mx$, we see that $v_-(x_1)\in\ker B_-(x_1,0)$ and hence $v(x_1)^*\Delta_w(x_1)v(x_1)=0$.
Therefore $T_0$ is unitarily equivalent to a problem posed on $(a,x_1)$ together with the boundary condition $u^-(x_1)\in\ker B_-(x_1,0)$.

Analogously, if $\ran\Delta_w(x_1)\subset \ran B_-(x_1,0)\neq\bb C^2$, we define
$$T_0=\{([(u\chi_{(a,x_1)})^\#],[f\chi_{(a,x_1]}]):(u,f)\in\mc T_\mx\}$$
and
$$T_1=\{([(u\chi_{(x_1,b)})^\#],[f\chi_{(x_1,b)}]):(u,f)\in\mc T_\mx\}$$
and find $T_\mx=T_0\oplus T_1$.
Here we have that $T_1$ is unitarily equivalent to a problem posed on $(x_1,b)$ together with the boundary condition $u^+(x_1)\in\ker B_+(x_1,0)$.

If $\Delta_w(x_1)=0$ (and $B_\pm(x_1,0)$ is not invertible) we have that both of the above factorizations work, i.e., $T_0$ and $T_1$ are both unitarily equivalent to a problem posed on appropriate subintervals of $(a,b)$  together with suitable boundary conditions.

\section{Weyl's alternative}\label{S:3}
Our first result in this section is a characterization of the closure of $T_\mn$.
\begin{thm}\label{T:3.1}
The closure $\ov{T_\mn}$ of the minimal relation $T_\mn$ is the set of all $([u],[f])\in T_\mx$ for which there is a balanced representative $u_0$ of $[u]$ such that $([u_0],[f])\in T_\mx$  and
$$(v^*Ju_0)^-(b)=(v^*Ju_0)^+(a)=0$$
for all $(v,g)\in\mc T_\mx$.
\end{thm}

\begin{proof}
If $([u_0],[f])\in T_\mx$ and $(v^*Ju_0)^-(b)=(v^*Ju_0)^+(a)=0$ whenever $(v,g)\in\mc T_\mx$, then in view of Lagrange's identity \eqref{Lagrange}, $([u_0],[f])\in T_\mx^*=\ov{T_\mn}$.

To prove the converse let $([u],[f])$ and $(v,g)$ be arbitrary elements of $\ov{T_\mn}$ and $\mc T_\mx$, respectively.
Again, Lagrange’s identity gives $(v^*Ju)^-(b)=(v^*Ju)^+(a)$ and we have to prove that one of these terms is $0$ at least for a specific choice $u_0\in [u]$ for $u$.

We assume, at first, that $0\not\in\Lambda$ and $\mc L_0=\{0\}$.
According to Lemma \ref{L:8.4mod1} we can find an element $(v_0,g_0)\in \mc T_\mx$ such that $v_0=v$ near $a$ and $v_0=0$ near $b$.
Hence
$$(v^*Ju)^+(a)=(v_0^*Ju)^+(a)-(v_0^*Ju)^-(b)=\<g_0,u\>-\<v_0,f\>=0.$$
This proves our claim in the first case.

Next we assume $\mc L_0\neq\{0\}$ but still $0\not\in\Lambda$.
In this case Lemma \ref{L:8.4mod2} gives an element $(v_0,g_0)\in \mc T_\mx$ such that $v_0=v$ near $a$ and $v_0=\gamma U(\cdot,0)e$ near $b$ for some number $\gamma$.
Additionally, we have the option of choosing a representative of $[u]$.
Lemma \ref{L:2.3} shows that $u(x)=\beta(x)U(x,0)e^\perp$ if $x$ is a point of continuity.
Here $\beta(x)=\beta(x_0)-e^*\int_{[x_0,x)}U(\cdot,0)^*wf$ has a limit as $x$ tends to $b$.
Since the elements of $\mc L_0$ are multiples of $U(\cdot,0)e^\perp$ we may choose a representative $u_0$ of $[u]$ for which $\lim_{x\uparrow b}\beta(x)=0$, i.e.,
we may assume, for points of continuity, $u_0(x)=\beta(x)U(x,0)e^\perp$ where $\beta(x)=e^*\int_{(x,b)}U(\cdot,0)^*wf$.
This gives
$$(v_0^*Ju_0)(x)=\ov{\gamma}\beta(x) e^*U(x,0)^*JU(x,0)e^\perp=-\ov{\gamma}\beta(x)$$
and hence $(v_0^*Ju_0)^-(b)=0$.
Thus
$$(v^*Ju_0)^+(a)=(v_0^*Ju_0)^+(a)-(v_0^*Ju_0)^-(b)=\<g_0,u_0\>-\<v_0,f\>=0$$
the desired conclusion.

Next we consider the case where $0$ is in $\Lambda$ but $\Lambda$ is not the whole complex plane.
In that case $\Lambda$ is countable and hence there are real numbers not in $\Lambda$.
Suppose $\mu$ is any such number, define $\tilde q=q-\mu w$ and $\tilde w=w$, and denote objects related to these new coefficients by a $\tilde{}$ accent.
Then note that $0\notin\tilde{\Lambda}$ and that $\tilde{\mc L}_0=\mc L_0$.
Since, according to \eqref{20230425.1}, $([u],[f-\mu u])\in\ov{\tilde T_\mn}$ we have, by what we already proved, an $u_0\in[u]$ such that $(\tilde v^*Ju_0)^+(a)=0$ whenever $(\tilde v,\tilde g)\in\tilde{\mc T}_\mx$.
In particular, we may choose $\tilde v=v$ and $\tilde g=g-\mu v$ to obtain $(v^*Ju_0)^+(a)=0$.

It now remains to consider the case when $\Lambda=\bb C$.
Suppose $x_1$ is a really bad point such that $\ran\Delta_w(x_1)\subset\ran B_+(x_1,0)$.
If $(v,g)\in\mc T_\mx$ we must have $v^-(x_1)\in\ker B_-(x_1,0)$.
To the right of $x_1$ we may continue $v$ and $g$ by $0$ and still have an element in $\mc T_\mx$.
This allows, as before, to show that $(v^*J u)^+(a)=0$.
If $\ran\Delta_w(x_1)\subset\ran B_-(x_1,0)$ we modify $(v,g)$ to the left of $x_1$ and can therefore show that $(v^*J u)^-(b)=0$.
\end{proof}

The following lemma outlines what happens when existence and uniqueness of solutions cannot be guaranteed.

\begin{lem}\label{P:3.3}
Suppose $\la\in\Lambda$ and $x_1,x_2\in\Xi_\la$ where $x_1\leq x_2$.
If $v$ is a solution of $Ju'+qu=\la wu$, then so are  $(v\chi_{(a,x_1)})^\#$, $(v\chi_{(x_2,b)})^\#$, and $(v\chi_{(x_1,x_2)})^\#$.
If $\la$ is not real, then $(v\chi_{(x_1,x_2)})^\#\in\mc L_0$ and if at least one of $\Lambda_{x_1}$ and $\Lambda_{x_2}$ is different from $\bb C$, we even have  $v\chi_{(x_1,x_2)}=0$.
\end{lem}

\begin{proof}
By (4) of Proposition \ref{P:2.1} we have $B_+(x_1,\la)u^+(x_1)=B_-(x_2,\la)u^-(x_2)=0$.
Therefore $\tilde v=(v\chi_{(x_1,x_2)})^\#$ is also a solution of $Ju'+qu=\la wu$.
Similarly for $(v\chi_{(a,x_1)})^\#$ and $(v\chi_{(x_2,b)})^\#$.

Since $\tilde v$ is in $\mc L^2(w)$, we have that $([\tilde v],\la[\tilde v])\in\ov{T_\mn}$.
But $\ov{T_\mn}$ cannot have non-real eigenvalues, so $\tilde v\in \mc L_0$, if $\lambda\in\bb C\setminus\bb R$.

Now suppose $\Lambda_{x_1}\neq\bb C$ and choose $\nu\in\bb C\setminus\Lambda_{x_1}$.
Then $J\tilde v'+q\tilde v=\nu w\tilde v$ is also satisfied in $(a,b)$.
Since initial value problems for $Ju'+qu=\nu wu$ have unique solutions and since $\tilde v$ vanishes to the left of $x_1$, it follows that $\tilde v$ is identically equal to $0$.
Of course the same argument applies when $\Lambda_{x_2}\neq\bb C$.
\end{proof}

While compactly supported solutions of $Ju'+qu=\la wu$ must be in $\mc L_0$ when $\la$ is not real, such solutions may have positive norm when $\la$ is real as the following example shows.
Let $(a,b)=(0,\infty)$ and
$$q=\sum_{n=1}^\infty \sm{2&2\\ 2&2}\delta_{2n-1}+\sum_{n=1}^\infty \sm{2&-2\\ -2&-2}\delta_{2n}\sas w=\sum_{n=1}^\infty \sm{2&0\\ 0&0}\delta_{2n-1}.$$
Then $B_+(2n-1,1)=B_-(2n,1)$ and $B_+(2n,1)=B_-(2n+1,1)$ and each of these fails to be invertible.
Therefore anyone of the intervals $(n,n+1)$, $n\in\bb N$ can support a compactly supported solution of $Ju'+qu=wu$ of positive norm and, consequently $1$ is an eigenvalue of $\ov{T_\mn}$ of infinite multiplicity.

Before we proceed we need to introduce some conventions.
If $\Xi_\la$ does not accumulate at the endpoint $b$, then there is a final interval $(\breve a,b)$ which does not intersect $\Xi_\la$ and has a regular endpoint $\breve a$.
We may choose $\breve a=\max\Xi_\la$ if $\Xi_\la\neq\emptyset$ and any point $\breve a\in(a,b)$ otherwise.
Let $\breve q$ and $\breve w$ be, respectively, the restrictions of $q$ and $w$ to the interval $(\breve a,b)$ and designate the corresponding relations and other objects associated with these new coefficients by a $\breve{}$ accent.
In particular, let $\breve{\mc D}_\la=\{(r,\la r)\in\breve{\mc T}_\mx\}$.
It will follow from Lemma~\ref{L:7.2} below that $\breve{\mc D}_\la$ is at least one-dimensional provided that $\la\not\in\bb R$.
Moreover, in Lemma~\ref{L:7.1}, we will show that $\breve{\mc D}_\la$ can be defined and is two-dimensional for some non-real $\la$ only if this is the case for all non-real $\la$.
Of course, we have an analogous situation if $\Xi_\la$ does not accumulate at $a$ in which case there is an initial interval $(a,\breve b)$ which does not intersect $\Xi_\la$ and has a regular endpoint $\breve b$.

\begin{lem}\label{L:7.2}
Suppose $\Lambda\neq\bb C$, $a$ is a regular endpoint, and $\la$ is not in $\bb R$.
Then there are non-trivial solutions of $Ju'+qu=\la wu$ in $\mc L^2(w)$.
The statement is also true when $b$ is a regular point.
\end{lem}

\begin{proof}
Clearly, we may assume that $\mc L_0=\{0\}$ and hence that $\mathbf B=\bb C^2$.
We also know that $\Lambda$ is at most countable and therefore its complement contains a real number $\mu$.
Let $U(\cdot,\mu)$ be a fundamental matrix for $Ju'+qu=\mu wu$ such that $U(x_0,\mu)=\id$ for some $x_0\in(a,b)$.
Lemma \ref{L:8.4mod1} gives elements $(u,f+\mu u)$ and $(v,g+\mu v)$ in $\mc T_\mx$ such that $u(a)=(1,0)^\top$, $v(a)=(0,1)^\top$, and $u(x)=v(x)=0$ near $b$.

Now suppose that the deficiency indices of $T_\mn$ are $0$ implying that $T_\mx$ is self-adjoint.
Since $(u,\mu u+f)$ and $(v,\mu v+g)$ are in $\mc T_\mx$ we get from Lagrange's identity \eqref{Lagrange} the absurdity that $0=\<v,\mu u+f\>-\<\mu v+g,u\> =v(a)^*Ju(a)=1$.
It follows that the deficiency spaces $D_\pm(\la)$, $\la\not\in\bb R$, cannot be trivial.
Thus there is a solution of $Ju'+qu=\la wu$ on $(a,b)$ with $\|u\|<\infty$.
A similar argument works, of course, with the roles of $a$ and $b$ reversed.
\end{proof}

\begin{lem}\label{L:7.1}
Suppose $a$ is a regular endpoint and, for some $\la_0\in\bb C\setminus\Lambda$, we have $\dim\mc D_{\la_0}=2$.
Then, for any $\la\in\bb C$ any solution of $Ju'+qu=\la wu$ is in $\mc L^2(w)$.
Moreover, if $\la$ is not real, then $\Xi_\la$ does not accumulate at $b$ and $\dim\mc D_{\la}=2$.
Again, the statement is true when the roles of $a$ and $b$ are reversed.
\end{lem}

\begin{proof}
We define the norm $\norm u_{c,d}$ by setting $\norm u_{c,d}=\int_{(c,d)} u^*wu$ whenever $c$ and $d$ are points of continuity.
Let $U(\cdot,\la_0)$ be the fundamental matrix for $Ju'+qu=\la_0 wu$ satisfying $U(a,\la_0)=\id$ and denote the first and second column of $U(\cdot,\la_0)$ by $\phi$ and $\psi$, respectively.
These are elements of $\mc L^2(w)$ and therefore we may choose $c$ such that $\norm{\phi}_{c,d}$ and $\norm{\psi}_{c,d}$ are smaller than $M=1/(2\sqrt{|\la-\la_0|})$ for all $d\in(c,b)$.

Now suppose $u$ satisfies $Ju'+qu=\la wu$ or, equivalently, $Ju'+(q-\la_0w)u=(\la-\la_0)wu$.
Then, by the variation of constants formula (Theorem \ref{t:vc}), there is a vector $(\alpha,\beta)^\top\in\bb C^2$ such that
$$u(x)=U(x,\la_0)\big(\sm{\alpha\\ \beta} +(\la-\la_0)J^{-1}\int_{(c,x)} U(\cdot,\ov{\la_0})^*wu\big)$$
whenever $x\geq c$ is a point of continuity.
Since $U(\cdot,\ov{\la_0})^*=(\ov{\phi},\ov{\psi})^*$ we introduce $\gamma(x)=\int_{(c,x)}\ov{\phi}{}^*wu$ and $\delta(x)=\int_{(c,x)}\ov{\psi}{}^* wu$.
Then the Cauchy-Schwarz inequality shows that $|\gamma(x)|\leq M \norm u_{c,d}$ and $|\delta(x)|\leq M \norm u_{c,d}$ whenever $x\in(c,d)$.
Next we get from Minkowski's inequality
$$\norm u_{c,d}\leq (|\alpha|+|\la-\la_0| M \norm u_{c,d})M +(|\beta|+|\la-\la_0| M \norm u_{c,d})M.$$
With our definition of $M$ this yields
$$\norm u_{c,d}\leq 2(|\alpha|+|\beta|)M.$$
Since the right hand side is independent of $d$ this implies that $u\in\mc L^2(w)$.

We have now shown that $\dim\mc D_\la=2$ whenever $\la\not\in\Lambda$.
In particular, $n_\pm=2-\dim\mc L_0$.
If $\mc L_0$ is not trivial Lemma \ref{L:2.4} shows that $\Lambda\subset\bb R$, i.e., $\dim\mc D_\la=2$ and $\Xi_\la=\emptyset$ for all non-real $\la$.
If, on the other hand, $\mc L_0=\{0\}$, we have $n_\pm=2$ and therefore two linearly independent solutions of $Ju'+qu=\la wu$ of positive finite norm exist for all non-real $\la$, even if $\la\in\Lambda\setminus\bb R$.
The only claim left to prove is that $\Xi_\la$ does not accumulate at $b$ when $\la\in\Lambda\setminus\bb R$.
If it did, Lemma \ref{P:3.3} would show that any solution of $Ju'+qu=\la wu$ vanished on $(\min{\Xi_\la},b)$ which would entail that $\dim\mc D_\la<2$, a contradiction.
\end{proof}

For certain kinds of endpoints the quantity $(v^*Ju)(x)$ tends to $0$ regardless of $u$ and $v$ in $\dom\mc T_\mx$ as $x$ approaches the endpoint.
Therefore boundary conditions cannot be posed at such endpoints or rather posing a boundary condition would be futile.
As we will prove below this happens for the endpoint $b$ if, for some non-real $\la$, either $\Xi_\la$ accumulates at $b$ or the space $\breve{\mc D}_\la$ is one-dimensional.
In this case we will say that we have the \emph{limit-point case at $b$}\index{limit-point case} or, for short, that $b$ is limit-point.
Otherwise, we say that we have the \emph{limit-circle case at $b$}\index{limit-circle case} or that $b$ is limit-circle.
Lemma \ref{L:7.1} (which is applicable since we may consider either the interval $(a,c)$ or else the interval $(c,b)$ for some $c\in (a,b)$)
shows that it does not depend on which non-real $\la$ is chosen to determine whether an endpoint is limit-point or limit-circle.
The terminology goes back to Weyl's famous 1910 paper \cite{41.0343.01}, but we alert the reader to the fact that its literal meaning applies only for $\la$ where $\Xi_\la$ is empty.
In \cite{BW} Weyl's point of view is explored further (even allowing for complex coefficients $q$ and $w$).
Of course, we have an analogous dichotomy for $a$ and, for any given problem, we have either the limit-point case or the limit-circle case at $a$.
Note that regular endpoints are limit-circle.

\begin{thm}\label{T:3.3}
If we have the limit-point case at $b$ and $([v],[g])\in T_\mx$, then for all $([u],[f])\in T_\mx$ there exists a balanced representative $u_0$ of $[u]$ such that $([u_0],[f])\in T_\mx$  and $(v^*Ju_0)^-(b)=0$.
An analogous statement holds when we have the limit-point case at $a$.
\end{thm}

\begin{proof}
First consider the case where $\Xi_\iu$ accumulates at $b$, let $c=\inf\Xi_\iu$, and consider a solution $r$ of $Ju'+qu=\iu wu$.
From Lemma \ref{P:3.3} we know that $(r\chi_{(c,b)})^\#$ is in $\mc L_0$.
If $a<c$ we must have $B_-(c,\iu)r^-(c)=0$ allowing for only one linearly independent solution on $(a,c)$.
Thus, if $r$ has finite positive norm, we have $n_+=1$.
Otherwise, or when $a=c$, we have $n_+=0$.
If $n_\pm=0$ our claim follows from Theorem \ref{T:3.1}.
Therefore we may assume now that $n_\pm=1$ so $D_\iu$ and $D_{-\iu}$ are one-dimensional.
Let $([r],\iu [r])$ be a basis vector of $D_\iu$, then $([\ov r],-\iu[\ov r])$ is a basis vector of $D_{-\iu}$ and we may assume $r=0$ on $(c,b)$.
Therefore, using von Neumann's relation \eqref{vNeumann}, there is a pair $(\tilde u,\tilde f)\in\mc T_\mx$ such that $([\tilde u],[\tilde f])\in\ov{T_\mn}$ and
$$(u,f)=(\tilde u+\alpha r+\beta\ov r,\tilde f+\alpha \iu r-\beta \iu \ov r)$$
where $\alpha$ and $\beta$ are appropriate numbers.
By Theorem \ref{T:3.1} there is a balanced function $\tilde u_0$ such that $\tilde u-\tilde u_0=\ell_0\in\mc L_0$ and $(v^*J\tilde u_0)^-(b)=0$.
Let $u_0=u-\ell_0$.
Then $([u_0],[f])\in T_\mx$  and $(v^*Ju_0)^-(b)=(v^*J\tilde u_0)^-(b)=0$.

Next we consider the case where $\Xi_\iu$ does not accumulate at $b$.
Let $\breve a=\max\Xi_\iu$ or, if $\Xi_\iu=\emptyset$, any point in $(a,b)$.
Denote the restriction of $(v,g)$ and $(u,f)$ to $(\breve a,b)$ by $(\breve v, \breve g)$ and $(\breve u, \breve f)$, respectively.
Of course, we have $(v^*Ju)^-(b)=(\breve v^*J\breve u)^-(b)$ and we consider now the problem on $(\breve a,b)$.
Since $\breve{\mc D}_\iu=\{r: (r,\iu r)\in\breve{\mc T}_\mx\}$ is one-dimensional we have that $\breve n_\pm$ is either $0$ or $1$.
In the former case $\dom \breve{\mc D}_\iu=\breve{\mc L}_0$, $\breve T_\mx=\ov{\breve T_\mn}$ and the claim follows from Theorem~\ref{T:3.1}.
The latter case, where $\breve{\mc L}_0=\{0\}$ and hence $\breve{\mathbf{B}}=\bb C^2$, needs some further attention.
Choose $\mu\in\bb R\setminus\breve\Lambda$ and let $r_1$ and $r_2$ be solutions of $Ju'+\breve qu=\mu \breve wu$ satisfying the initial conditions $r_1^+(\breve{a})=(1,0)^\top$ and $r_2^+(\breve a)=(0,1)^\top$, respectively.
By Lemma \ref{L:8.4mod1} there are elements $(\tilde r_j,\tilde h_j)\in\breve{\mc T}_\mx$, $j=1,2$, such that
$\tilde r_j=r_j$ near $\breve a$ and $\tilde r_j=0$ near $b$.
Since $(\tilde r_2^*J\tilde r_1)^+(\breve a)=1$ Theorem \ref{T:3.1} shows that neither $([\tilde r_1],[\tilde h_1])$ nor $([\tilde r_2],[\tilde h_2])$ is in $\ov{\breve{T}_\mn}$.
Therefore, by von Neumann's relation \eqref{vNeumann}
$$(\breve u,\breve f)=(\tilde u+\alpha\tilde r_1+\beta\tilde r_2,\tilde f+\alpha\tilde h_1+\beta\tilde h_2)$$
where $([\tilde u],[\tilde f])\in\ov{\breve{T}_\mn}$.
Lemma \ref{L:8.4mod1} is also used to find an element $(\tilde v,\tilde g)\in \breve{\mc T}_\mx$ such that $\tilde v=0$ near $\breve a$ and $\tilde v=\breve v=v$ near $b$.
Then, using Lagrange's identity \eqref{Lagrange},
\[(v^*Ju)^-(b)=(\breve v^*J\breve u)^-(b)=(\tilde v^*J\tilde u)^-(b)-(\tilde v^*J\tilde u)^+(a)=\scp{\tilde v,\tilde f}-\scp{\tilde g,\tilde u}=0.\qedhere\]
\end{proof}

We now use all the information gathered to investigate Weyl's alternative under various circumstances.

\subsection{The definite case, no really bad points are present}\label{S:3.1}
Here we have a trichotomy which takes the same form as in the classical case.

\begin{thm}\label{T:8.7}
Suppose $\Lambda\neq\bb C$ and $\mc L_0=\{0\}$.
Then the following statements are true.
  \begin{enumerate}
    \item $n_+=n_-=0$ if and only if we have the limit-point case at both $a$ and $b$.
    In this situation $T_\mx$ is self-adjoint.
    \item $n_+=n_-=1$ if and only if we have the limit-point case at one of $a$ and $b$ and the limit-circle case at the other.
    Self-adjoint restrictions are given by posing a separated boundary condition at the limit-circle end.
    \item $n_+=n_-=2$ if and only if we have the limit-circle case at both $a$ and $b$.
  \end{enumerate}
\end{thm}

\begin{proof}
Suppose we have the limit-point case at both $a$ and $b$.
Then Theorems \ref{T:3.3} and \ref{T:3.1} give $\ov{T_\mn}=T_\mx$.
Hence $T_\mx$ is self-adjoint and the deficiency indices of $T_\mn$ are zero.
Conversely, if at least one of $a$ and $b$ is limit-circle and since $\mc L_0$ is trivial, a slight modification of Lemma \ref{L:7.2} gives that $n_\pm>0$.
This proves (1).

Now suppose that we have the limit-circle case at both $a$ and $b$.
Then every solution of $Ju'+qu=\pm\iu w u$  is in $\mc L^2(w)$, i.e., $n_\pm=2$.
Conversely, if at least one of $a$ and $b$ is limit-point we must have $n_\pm<2$.
This proves (3) and thus also the first statement of (2).

For the second statement of (2) assume that the boundary condition is given by $(v,g)\in D_\iu\oplus D_{-\iu}$ and assume that $a$ is limit-circle and $b$ is limit-point (the other case being treated similarly).
Since $(g,-v)$ is also in $D_\iu\oplus D_{-\iu}$ Theorem \ref{T:3.3} gives $(g^*Ju)^-(b)=0$ for any $(u,f)\in T_\mx$ so that the boundary condition posed is $(g^*Ju)^+(a)=0$.
\end{proof}

\subsection{The non-definite case, no really bad points are present}\label{S:3.2}
Our previous considerations give us immediately the following result.
\begin{thm}\label{T:8.7L0}
Suppose $\Lambda\neq\bb C$ and $\dim\mc L_0=1$.
Then $n_+=n_-=0$ if and only if we have the limit-point case at $a$ or $b$ and $n_+=n_-=1$ if and only if we have the limit-circle case at both $a$ and $b$.
\end{thm}

It was shown in Theorem 7.3 of \cite{MR4047968} that $T_\mx=\{0\}\times L^2(w)$ when $\dim\mc L_0=1$, $n_\pm=0$, and $\ker\Delta_w(x)\subset \ker\Delta_w(x)J^{-1}\Delta_q(x)$ for all $x\in (a,b)$.
This happens, in particular, when $w$ is a continuous measure but also when $\Delta_w(x)$ is invertible whenever it is different from $0$.
However, if $\Delta_w(x)$ has rank $1$ for some $x\in (a,b)$ we may have a different situation. as may be seen in the example where $(a,b)=(0,\infty)$, $J=\sm{0&-1\\ 1&0}$, $q=\sm{0&0\\ 0&2}\sum_{k=1}^\infty(-1)^k\delta_k$, and $w=\sm{2&0\\ 0&0}\sum_{k=1}^\infty \delta_k$.
In this case $T_\mx$ is, in fact, a self-adjoint infinite-dimensional operator.

\subsection{Really bad points are present}\label{S:3.3}
If we have the limit-point case at both $a$ and $b$, it is still the case that $n_+=n_-=0$ and that $T_\mx$ is self-adjoint.
If, however, we have the limit-point case only at one endpoint but not the other we can only conclude that $n_\pm\leq 1$.
Similarly, knowing only that we have the limit-circle case at both endpoints gives no information on the deficiency indices.
The additional information required concerns the structure of $\mc L_0$.

Assume that $a$ is a limit-circle endpoint and that $b$ is a limit-point endpoint.
There is then a smallest really bad point which we call $x_1$.
Pick a non-real number $\la$ such that $B_\pm(x,\la)$ are invertible for all $x\in(a,x_1)$ (recall that all but countable many of the non-real $\la$ are of this kind) and let $v_0$ be a non-trivial solution of the equation $Ju'+qu=\la wu$ which satisfies the boundary condition $B_-(x_1,\la)u^-(x_1)=0$ and vanishes on $(x_1,b)$.
If $(u,\la u)\in\mc T_\mx$ let $u_0=(u \chi_{(a,x_1)})^\#$ and $u_1=(u \chi_{(x_1,b)})^\#$.
Using that $u_1=0$ near $a$, Theorem \ref{T:3.3} and Theorem \ref{T:3.1} show that $([u_1],\la [u_1])\in\ov{T_\mn}$.
Since $\la$ cannot be an eigenvalue of $\ov{T_\mn}$ we have that $u_1\in\mc L_0$.
Therefore $\norm u\neq0$ if and only if $\norm{u_0}\neq0$.
But $u_0$ is a constant multiple of $v_0$.
Thus, if $v_0$ has norm $0$, then $\dom\mc D_\la\subset\mc L_0$, $D_\la$ is trivial, and $n_\pm=0$.
Otherwise, if $v_0$ has positive norm, then $n_\pm=1$.

The situation is analogous when $b$ is a limit-circle endpoint and really bad points accumulate at $a$.

Finally, if both $a$ and $b$ are limit-circle endpoints we denote the smallest and largest really bad point by $x_1$ and $x_N$, respectively.
As before we choose a non-trivial solution $v_0$ of $Ju'+qu=\la wu$ which vanishes on $(x_1,b)$ and satisfies the boundary condition $B_-(x_1,\la)u^-(x_1)=0$.
Similarly, $v_N$ is a non-trivial solution of $Ju'+qu=\la wu$ vanishing on $(a,x_N)$ and satisfying the boundary condition $B_+(x_N,\la)u^+(x_N)=0$.
Now we have the following trichotomy:
\begin{enumerate}
\item $n_\pm=2$ if and only if both $v_0$ and $v_N$ have positive norm.
\item $n_\pm=1$ if and only if precisely one of $v_0$ and $v_N$ has positive norm (and the other norm $0$).
\item $n_\pm=0$ if and only if both $v_0$ and $v_N$ have norm $0$.
\end{enumerate}

Examples show that all these cases may actually occur.

\section{Boundary conditions}\label{S:4}
When $n_\pm=0$ the relation $T_\mx$ is self-adjoint and no boundary conditions can be (or need to be) posed.
If, however $n_\pm>0$ the relation $T_\mx$ has infinitely many distinct self-adjoint restrictions determined by boundary conditions as stated in Section \ref{S:2.2}:
any self-adjoint restriction of $T_\mx$ is determined by $n_\pm$ linearly independent elements $(v_1,g_1), ..., (v_{n_\pm},g_{n_\pm})\in D_\iu\oplus D_{-\iu}$ satisfying $(g_j^*Jg_k)^-(b)-(g_j^*Jg_k)^+(a)=0$, $j,k=1,..., n_\pm$.
Conversely, if these conditions hold, then
$$T=\{([u],[f])\in T_\mx: \forall j=1,...,n_\pm: (g_j^*Ju)^-(b)-(g_j^*Ju)^+(a)=0\}$$
is a self-adjoint restriction of $T_\mx$.
The purpose of this section is to give a more classical representation of the boundary conditions.

If $a$ is a limit-circle endpoint, then really bad points cannot accumulate at $a$.
Therefore there is a real number $\mu$ and an interval $(a,\breve b)$ such that $B_\pm(x,\mu)$ are invertible for all $x\in(a,\breve b)$.
For the equation $Ju'+qu=\mu wu$ posed on this interval we choose a balanced fundamental matrix $U_\ell(\cdot,\mu)$ which is symplectic at some point of continuity in $(a,\breve b)$ (and hence at all such points).
For $u\in\dom\mc T_\mx$ and $x\in(a,\breve b)$ we define now $\cev u(x)=U_\ell(x,\mu)^*Ju(x)$.
Since $a$ is limit-circle it follows from Lagrange's identity \eqref{pLagrange} that $\cev u(a)=\lim_{x\downarrow a}\cev u(x)$ exists.
Similarly, if $b$ is limit-circle, we have a balanced fundamental matrix $U_r(\cdot,\mu)$, symplectic at points of continuity, on some interval $(\breve a,b)$ and using it we define a map $\vec u=U_r(\cdot,\mu)^*Ju$ on $(\breve a,b)$ for any $u\in\dom\mc T_\mx$.
Again $\vec u(b)=\lim_{x\uparrow b}\vec u(x)$ exists.
We will call the vectors $\cev u(a)$ and $\vec u(b)$ quasi-boundary values of $u\in\dom\mc T_\mx$.
If both $a$ and $b$ are limit-circle endpoints, it is, of course, possible to choose a common $\mu$.
If $\Xi_\mu=\emptyset$ we may even choose $U_r=U_\ell$ in which case we have $\vec u=\cev u$ and will then write $\cevec u$ for either.

We emphasize that, if $a$ is a regular endpoint, the canonical choice for $U_\ell(\cdot,\mu)$ is the one satisfying $U_\ell(a,\mu)=J$ in which case the quasi-boundary values are equal to the actual boundary values.
Similarly, if $b$ is a regular endpoint, we choose $U_r(\cdot,\mu)$ so that $U_r(b,\mu)=J$ and hence $\vec u(b)=u(b)$.

We prove now a key lemma.
\begin{lem}\label{L:4.1}
Suppose $Ju'+qu=wf$ and $Jv'+qv=wg$ on $(a,\breve b)$.
Then $v(x)^*Ju(x)=\cev v(x)^*J\cev u(x)$ for all $x\in(a,\breve b)$.
Similarly, if $Ju'+qu=wf$ and $Jv'+qv=wg$ on $(\breve a,b)$, then $v(x)^*Ju(x)=\vec v(x)^*J\vec u(x)$ when $x\in(\breve a,b)$.
\end{lem}

\begin{proof}
Let $U=U_\ell$ or $U=U_r$.
Then multiply identity \eqref{wronski} for $\la=\mu\in\bb R$ by $JU(x,\mu)J$ on the left and by $U(x,\mu)^{-1}$ on the right to obtain the identity $JU(x,\mu)JU(x,\mu)^*J=-J$ for points of continuity.
Multiplying this new identity with $v(x)^*$ on the left and with $u(x)$ on the right proves our claim.
\end{proof}

\subsection{Deficiency indices are \texorpdfstring{$\mathbf{2}$}{2}}\label{S:4.1}
Suppose $T$ is a self-adjoint restriction of $T_\mx$ determined by $(v_1,g_1),(v_2,g_2)\in D_\iu\oplus D_{-\iu}$, i.e.,
\begin{equation}\label{saT}
T=\{([u],[f])\in T_\mx: (g_j^*Ju)^-(b)-(g_j^*Ju)^+(a)=0 \text{ for $j=1,2$}\}.
\end{equation}

Since $n_\pm=2$ our results in Sections \ref{S:3.1} -- \ref{S:3.3} show that $a$ and $b$ are limit-circle endpoints.
If there are no very bad points $\mc L_0$ must be trivial and there is then a unique $u\in[u]$ satisfying $Ju'+qu=wf$ for any element $([u],[f])\in T_\mx$.
This allows us to uniquely define quasi-boundary values $\cev u(a)$ and $\vec u(b)$ for any $u\in\dom T_\mx$.
If there are really bad points these cannot accumulate at either $a$ or $b$.
According to Section \ref{S:3.3} the restrictions of $u$ to the first and last intervals are uniquely determined for any element $([u],[f])\in T_\mx$ so that again quasi-boundary values are well-defined.
Using Lemma \ref{L:4.1} we may now write
$$T=\{([u],[f])\in T_\mx: A\sm{\cev u(a)\\ \vec u(b)}=0\}$$
where
\begin{equation}\label{aaa}
A=\bsm{\cev g_1(a)^*&\vec g_1(b)^*\\ \cev g_2(a)^*&\vec g_2(b)^*}\bb J\in\bb C^{2\times 4}
\end{equation}
and $\mathbb{J}=\sm{-J&0\\0&J}$.

We now need the following lemma.
\begin{lem}\label{qbv}
The map
\[\mc B:D_\iu\oplus D_{-\iu}\to\mathbb{C}^4:(v,g)\mapsto\begin{bmatrix}\cev{v}(a)\\ \vec{v}(b)\end{bmatrix}\]
is a linear bijection.
\end{lem}

\begin{proof}
If $\cev v(a)=0$, then $(u^*Jv)^+(a)=\cev{u}(a)^*J\cev{v}(a)=0$ for any $(u,f)\in \mc T_{\mx}$.
Similarly, if $\vec v(b)=0$ we get $(u^*Jv)^-(b)=0$ whenever $(u,f)\in \mc T_{\mx}$.
According to Theorem \ref{T:3.1}, it follows that $(v,g)\in\overline{T_\mn} $ and hence $(v,g)=0$, i.e., $\ker\mc B=\{0\}$.
\end{proof}

The lemma shows that, since $(v_1,g_1)$ and $(v_2,g_2)$ are linearly independent, the matrix $A$ in \eqref{aaa} has rank $2$.
The condition that $(g_j^*Jg_k)^-(b)-(g_j^*Jg_k)^+(a)=0$ for $j,k=1,2$ implies that $A\bb JA^*=0$.

Conversely, given a $2\times 4$-matrix $A$ of rank $2$ satisfying $A\bb JA^*=0$ as well as maps $u\mapsto \cev u$ and $u\mapsto \vec u$ as described above, one can construct a self-adjoint restriction of $T_\mx$ as we prove next.
Lemma \ref{qbv} shows that there are elements $g_1, g_2\in\dom(D_\iu\oplus D_{-\iu})$ with its quasi-boundary values given by the columns of the matrix
$$\bb J A^*=\bsm{\cev g_1(a)&\cev g_2(a)\\ \vec g_1(b)&\vec g_2(b)},$$
respectively.
Since $\ran(D_\iu\oplus D_{-\iu})=\dom(D_\iu\oplus D_{-\iu})$ we have therefore two elements $(v_1,g_1)$ and $(v_2,g_2)\in D_\iu\oplus D_{-\iu}$.
These are linearly independent on account of $A$ having rank $2$ and they satisfy $(g_j^*Jg_k)^-(b)-(g_j^*Jg_k)^+(a)=0$, $j,k=1,2$ since $A\bb JA^*=0$.
Therefore the relation $T$, given as in \eqref{saT}, is a self-adjoint restriction of $T_\mx$.

We now have a closer look at the structure of $A$ which we write as a block matrix with two $2\times2$-blocks, i.e., $A=(A_1,A_2)$.
The condition $A\bb JA^*=0$ becomes $A_1JA_1^*=A_2JA_2^*$ which implies that $A_1$ is invertible if and only if $A_2$ is invertible.
Thus we have to discuss two cases, namely (i) $A_1$ and $A_2$ are both invertible and (ii) neither is invertible.
But before we do so let us note that multiplying $A$ from the left by a constant invertible $R$ gives a new matrix $RA$ of rank $2$ satisfying $RA\bb J(RA)^*=0$.
If $\bb J A^*$ determines the elements $(v_1,g_1)$ and $(v_2,g_2)\in D_\iu\oplus D_{-\iu}$, then $\bb J(RA)^*$ leads to linear combinations of these.
It follows that $A$ and $RA$ give rise to the same self-adjoint restriction of $T_\mx$.

If $A_1$ and $A_2$ are not invertible we choose $R$ so that $RA=\sm{c_1&c_2&0&0\\ 0&0&c_3&c_4}$.
Then $RA\bb J(RA)^*=0$ implies that $c_1\ov c_2$ and $c_3\ov c_4$ must be real.
Therefore we may as well assume that $A=\sm{\cos\alpha&\sin\alpha&0&0\\ 0&0&\cos\beta&\sin\beta}$ for appropriate $\alpha,\beta\in[0,\pi)$.
This is the case of separated boundary conditions.
Clearly $\alpha$ and $\beta$ are uniquely determined by $T$.

If $A_1$ and $A_2$ are invertible we may choose $R=A_2^{-1}$ so that $RA=(S,\id)$ with a (complex) symplectic matrix $S$.
This is the case of coupled boundary conditions.
We now show that $(S_1,\id)$ and $(S_2,\id)$ determine the same $T$, as given by \eqref{saT}, only if $S_1=S_2$.
Since $g_1$ and $g_2$ are in the domain of $T$ we get $\vec g_j(b)+S_k\cev g_j(a)=0$ for $j,k=1,2$.
Assume that $\cev g_1(a)$ and $\cev g_2(a)$ are linearly dependent, so that there are numbers $c_1$ and $c_2$ such that $c_1\cev g_1(a)+c_2\cev g_2(a)=0$.
Applying $S_1$ shows that $c_1\vec g_1(b)+c_2\vec g_2(b)=0$.
But this is impossible since $(S_1,\id)$ has rank $2$, and therefore $\cev g_1(a)$ and $\cev g_2(a)$ are actually linearly independent.
Next we have $S_2\cev g_j(a)=-\vec g_j(b)=S_1\cev g_j(a)$.
This means that $\ker(S_2-S_1)$ is two-dimensional, i.e., $S_2=S_1$.

\subsection{Deficiency indices are \texorpdfstring{$\mathbf{1}$}{1}; coupled boundary conditions}\label{S:4.2}
Since $n_\pm=1$ we must have at least one limit-circle endpoint.
If the other endpoint is limit-point, then only a separated boundary condition can be posed.
Thus, to have $n_\pm=1$ and coupled boundary conditions both endpoints have to be limit-circle.
If there are really bad points, recall the definitions of $v_0$ and $v_N$ in Section \ref{S:3.3}.
Precisely one of them has positive norm, say $v_0$, and that means that the spanning element of $\dom D_\iu$ may be chosen with support near $a$.
Hence, any boundary condition will involve only the endpoint $a$, i.e., it would be a separated boundary condition.
If $v_N$ has positive norm instead, we obtain again only a separated boundary conditions.
It follows from these considerations that $n_\pm=1$ and coupled boundary conditions occur only when we have two limit-circle endpoints, no really bad points, and $\mc L_0\neq\{0\}$.

Fix a point $x_0\in(a,b)$ and a $\mu\in\bb R\setminus\Lambda$.
We choose $\varphi\in\mc L_0$ so that $\varphi(x_0)$ is a unit vector in $\bb R^2$ and we define $\psi$ as the solution of the initial value problem $J\psi'+q\psi=\mu w\psi$, $\psi(x_0)=J\varphi(x_0)$.
Then $\Phi=(\varphi,\psi)$ is a fundamental matrix for $Ju'+qu=\mu wu$ satisfying $\Phi^*J\Phi=J$ at points of continuity.
We use this fundamental matrix to define the map $u\mapsto \cevec u=\Phi^*Ju$ for $u\in\dom\mc T_\mx$.
In particular, $\cevec\Phi(x)=J$, i.e., $\cevec\varphi(x)=\sm{0\\ 1}$ and $\cevec\psi(x)=\sm{-1\\ 0}$.
As these do not depend on $x$ we suppress their arguments below.

We need the following lemma.
\begin{lem}\label{L:4.4}
If $(u,f)\in\mc T_\mx$, then
$$\cevec u^-(x)=\cevec u(a)+\int_{(a,x)}\psi^* w(f-\mu u)\cevec\varphi.$$
In particular, the first component of $\cevec u$ is constant at points of continuity.
\end{lem}

\begin{proof}
Since $\varphi\in\mc L_0$ we have $w\varphi=0$ and $(\varphi,0)\in\mc T_\mx$.
Thus Lagrange's identity \eqref{pLagrange} shows that $(\varphi^*Ju)^-(x)-(\varphi^*Ju)^+(a)=\int_{(a,x)}\varphi^*wf=0$.
This is the first component of $\cevec u^-(x)-\cevec u(a)$.
We get for the second component $\int_{(a,x)}\psi^* w(f-\mu u)$ when we use \eqref{pLagrange} for $(\psi,\mu\psi)$ and $(u,f)$.
\end{proof}

Suppose $([u],\iu[u])$ spans $D_\iu$ (which has dimension $1$).
Then we have, by the previous lemma, that $\cevec u(x)=\alpha(x)\cevec\varphi+\beta\cevec\psi$ for an appropriate function $\alpha$ and a number $\beta$ at points of continuity.
Since, by Lagrange's identity \eqref{Lagrange} and Lemma \ref{L:4.1}, we have
\begin{equation}\label{nu}
0\neq 2\iu\norm{u}^2=\cevec u(b)^*J\cevec u(b)-\cevec u(a)^*J\cevec u(a)=2\iu \Im((\alpha(b)-\alpha(a))\ov\beta),
\end{equation}
we find that $\beta\neq0$.
Thus, defining the function $v_+=(u-\alpha\varphi)/\beta$, we have that $([v_+],\iu[v_+])$ is also an element which spans $D_\iu$.
Note that $\cevec v_+(a)=\cevec\psi=\sm{-1\\ 0}$ and $\cevec v_+(b)=p\cevec\varphi+\cevec\psi=\sm{-1\\ p}$ where $p=(\iu-\mu)\<\psi,v_+\>$ according to Lemma \ref{L:4.4}.
Since, as in \eqref{nu}, $0<\norm{v_+}^2=\Im(p)$ we obtain that $p$ cannot be real.
If we set $v_-=\ov{v_+}$ we get that $([v_-],-\iu[v_-])$ spans $D_{-\iu}$.
Also, $\cevec v_-(a)=\cevec\psi$ and $\cevec v_-(b)=\ov{p}\cevec\varphi+\cevec\psi$.
Thus we constructed a linear bijection between $D_\iu\oplus D_{-\iu}$ and
$$\mc D=\{(c_1 v_++c_2v_-,\iu(c_1 v_+-c_2 v_-)):c_1,c_2\in\bb C\}.$$
We now define the quasi-boundary map
$$\mc B:\mc D\to\bb C^4:(v,g)\mapsto\begin{bmatrix}\cevec{v}(a)\\ \cevec{v}(b)\end{bmatrix}.$$
In particular, $\mc B(c_1 v_++c_2v_-,\iu(c_1 v_+-c_2 v_-))=(-c_1-c_2,0,-c_1-c_2,c_1 p+c_2\ov p)^\top$.
Since $p$ is not real, this set of quasi-boundary values has dimension $2$ and is therefore in a linear one-to-one correspondence with $\mc D$ and thus with $D_\iu\oplus D_{-\iu}$.

Now let $T$ be a self-adjoint restriction of $T_\mx$ determined by the non-trivial element $(v,g)\in\mc D$, i.e.,
\begin{multline*}
T=\{([u],[f])\in T_\mx: (g^*Ju)^-(b)-(g^*Ju)^+(a)=0\}\\
 =\{([u],[f])\in T_\mx: \cevec g(b)^*J\cevec u(b)-\cevec g(a)^*J\cevec u(a)=0\}.
\end{multline*}
Recall that  $0=(g^*Jg)^-(b)-(g^*Jg)^+(a)=\cevec g(b)^*J\cevec g(b)-\cevec g(a)^*J\cevec g(a)$.
If we set $A=(-\cevec g(a)^*J,\cevec g(b)^*J)\in\bb C^{1\times 4}$, we get
$$T=\{([u],[f])\in T_\mx: A\sm{\cevec u(a)\\ \cevec u(b)}=0\}.$$
Note that $\sm{\cevec g(a)\\ \cevec g(b)}=\bb JA^*$ is an element of $\ran\mc B$.
Therefore we have that $0\neq A\in\operatorname{span}\{(0,-1,0,1),(0,0,1,0)\}$ and $A\bb JA^*=0$.
If $A=(0,-s,t,s)$ we get $A\bb JA^*=2\iu \Im(t\ov s)$, i.e., $t\ov s$ must be real.
Since we may multiply $A$ by a non-zero constant without changing the boundary condition we may assume that $A=(0,-\sin\alpha,\cos\alpha,\sin\alpha)$ for some $\alpha\in[0,\pi)$.

Conversely, suppose we have $A=(0,-\sin\alpha,\cos\alpha,\sin\alpha)$.
Then we can find a non-trivial element $(g,-v)\in \mc D$ whose quasi-boundary values are $\bb JA^*$.
It also satisfies $(g^*Jg)^-(b)-(g^*Jg)^+(a)=0$ and this implies that the relation
$$T=\{([u],[f])\in T_\mx:A\sm{\cevec u(a)\\ \cevec u(b)}=0\}$$
is self-adjoint.

If both endpoints are regular we could unlink the quasi-boundary values by using fundamental matrices $\Phi_\ell$ and $\Phi_r$ satisfying $\Phi_\ell(a)=\Phi_r(b)=J$.
The above considerations would still go through and the quasi-boundary values would coincide with the actual boundary values.

\subsection{Deficiency indices are \texorpdfstring{$\mathbf{1}$}{1}; separated boundary condition}\label{S:4.3}
This situation occurs in two instances which we describe now.
\begin{enumerate}
  \item One endpoint is limit-circle while the other one is limit-point (whether or not really bad points are present), see Theorem \ref{T:3.3}.
As we will see the separated boundary is then to be posed at the limit-circle end.
  \item Both endpoints are limit-circle in the presence of (necessarily finitely many) really bad points.
Recall the functions $v_0$ and $v_N$ defined in Section \ref{S:3.3}.
If both have norm $0$, then $n_\pm=0$, a case which we do not need to consider.
If both have positive norm, then $n_\pm=2$, a case which was considered in Section \ref{S:4.1}.
However, if precisely one of $v_0$ and $v_N$ has positive norm (and the other norm $0$), we will need a separated boundary condition.
It is posed at $a$ if $\norm{v_0}>0$ and at $b$ if $\norm{v_N}>0$.
\end{enumerate}

We will assume in the following that $b$ is the endpoint in whose vicinity all solutions of $Ju'+qu=\la wu$ have either zero or infinite norm.
The other case is, of course, treated similarly.

If $T$ is a self-adjoint restrictions of $T_\mx$ determined by the element $(v,g)\in D_\iu\oplus D_{-\iu}$, then $(g^*Ju)^-(b)=0$ according to Theorem \ref{T:3.3} (recall that $(g,-v)$ is also in $D_\iu\oplus D_{-\iu}\subset T_\mx$).
Therefore
$$T=\{([u],[f])\in T_\mx: (g^*Ju)^+(a)=0\}$$
and, using Lemma \ref{L:4.1},
$$T=\{([u],[f])\in T_\mx: \cev g(a)^*J\cev u(a)=0\}$$
after choosing an appropriate fundamental matrix as explained above.
Let $A^*=J^*\cev g(a)$, an element of $\bb C^2$.
Then $AJA^*=0$ but $A\neq0$.
The former condition means that we may choose $A^*\in\bb R^2$ since we may multiply $g_1$ by a nonzero constant without changing the boundary condition.
Indeed, we may choose $A=(\cos\alpha,\sin\alpha)$ for some $\alpha\in[0,\pi)$.

Conversely, if $A=(\cos\alpha,\sin\alpha)$ with $\alpha\in[0,\pi)$, we have $A\neq0$, $AJA^*=0$.
The same argument used to prove Lemma \ref{qbv} shows that the map $\mc B:D_\iu\oplus D_{-\iu}\to\bb C^2:(v,g)=\cevec v(a)$ is again a linear bijection.
Therefore $A$ determines a non-trivial element $(v,g)\in D_\iu\oplus D_{-\iu}$ satisfying $(g^*Jg)^+(a)=0$.

Thus we have shown that all self-adjoint restrictions of $T_\mx$ are given by
$$T=\{([u],[f])\in T_\mx: (\cos\alpha,\sin\alpha)\cev u(a)=0\}$$
where $\alpha\in[0,\pi)$.
If $a$ is regular we may replace $\cev u(a)$ by the true boundary values $u(a)$.

\section{The Krein-von Neumann extension}\label{S:5}
In this section we are studying the Krein-von Neumann extension for non-negative relations associated with the equation $Ju'+qu=wf$ in the case of limit-circle endpoints.
We begin with providing some basic facts about abstract Krein-von Neumann extensions.

Suppose $\mc H$ is a Hilbert space with inner product $\scp{\cdot,\cdot}$ and $S$ is a symmetric linear relation in $\mc H\times\mc H$.
$S$ is called bounded below (by $\alpha$), if there is an $\alpha\in\bb R$ such that $\<f,u\>\geq \alpha\norm{u}^2$ for all $(u,f)\in S$.
$S$ is called non-negative if one can choose $\alpha=0$.

One convenient definition of the Krein-von Neumann extension $S_{\rm KvN}$ of the non-negative relation $S$ is the inverse of the Friedrichs extension of $S^{-1}$, i.e.,
$$S_{\rm KvN}= ((S^{-1})_F)^{-1},$$
where the subscript $F$ denotes the Friedrichs extension.
This definition relies on the fact that $S$ is non-negative if and only if $S^{-1}$ is non-negative.
For background information on such matters we refer the reader, for example, to Riesz and Sz.-Nagy \cite{MR1068530}, Ando and Nishio \cite{MR0264422}, Faris \cite{MR0467348}, Behrndt, Hassi, and de Snoo \cite{MR3971207}, and Fucci et al. \cite{MR4410827}.

The Friedrichs extension for linear operators can be found in many books.
For linear relations there are not so many sources, but Behrndt, Hassi, and de Snoo \cite{MR3971207} is one of them.
In any case we add the Appendix \ref{A:A} for the convenience of the reader and to present our particular viewpoint.

Combining Theorem \ref{Friedrichs} with the above definition of the Krein-von Neumann extension leads to the following equivalent formulation.
\begin{dfn}
Suppose $S$ is a non-negative relation in $\mc H\times\mc H$.
Then
\[S_{\rm KvN}=\{(u,f)\in S^*: \exists (u_n,f_n)\in S:f_n\to f\text{ and }\langle u_n,f_n\rangle\to\langle u,f\rangle\}\]
is called the Krein-von Neumann extension\footnote{Clearly, $S$ and even $\ov S$ are contained in $S_{\rm KvN}$.} of $S$.
\end{dfn}

We obtain immediately the following result.
\begin{prop}\label{P:6.2}
Suppose $S$ is a non-negative relation in $\mc H\times\mc H$.
Then  we have $\ker S^*\times\{0\}\subset S_{\rm KvN}$.
\end{prop}

The simplest example of a Krein-von Neumann extension is for $S=\{(0,0)\}$.
In this case $S^*=\mc H\times\mc H$ and $S_{\rm KvN}=\ker S^*\times\{0\}=\mc H\times\{0\}$.

We now apply these abstract definitions in the case when $S=T_\mn$, the minimal relation associated with the equation $Ju'+qu=wf$.
We set $T_K=(T_\mn)_{\rm KvN}$.
Our goal is to identify the boundary conditions determining $T_K$.

First assume that $0\notin\Lambda$ and choose a fundamental matrix $U(\cdot,0)$ for $Ju'+qu=0$ satisfying the requirements of Section \ref{S:4}.
Using this fundamental matrix we then define the quasi-boundary values $\cevec{u}(a)$ and $\cevec{u}(b)$ as we did there; recall that $\cevec u(x)=U(x,0)^*Ju(x)$ .

\begin{thm}\label{T:6.3}
Suppose the endpoints $a$ and $b$ are limit-circle, $T_\mn$ is non-negative, and $0\notin\Lambda$.
Then the Krein-von Neumann extension is given by
\begin{multline}\label{eqtk}
T_K=\ov{T_\mn}\dirsum(\ker T_\mx\times\{0\})=\{(u,f)\in T_\mx:\cevec{u}(b)=\cevec{u}(a)\}\\
 =\{(u,f)\in T_\mx:\int U(\cdot,0)^*wf=0\}.
\end{multline}
In particular, $T_K$ is self-adjoint.
\end{thm}

We emphasize that the Krein-von Neumann extension is determined by periodic quasi-boundary values for any choice of a fundamental matrix that is symplectic at some (and hence all) point(s).

\begin{proof}
The last equality in \eqref{eqtk} follows from Lagrange's identity \eqref{Lagrange}.
Note that no non-trivial element of $\ker\mc T_\mx$ can be compactly supported and hence that the intersection of $\ov{T_\mn}$ and $\ker{T_\mx}\times\{0\}$ is trivial.
Now define
$$R=\ov{T_\mn}\dirsum(\ker T_\mx\times\{0\})\quad\text{and}\quad T=\{(u,f)\in T_\mx:\cevec{u}(b)=\cevec{u}(a)\}.$$
Then Proposition \ref{P:6.2} gives that $R\subset T_K$.
We proved in Section \ref{S:4.1} that $T$ is a self-adjoint extension of $T_\mn$ when $n_\pm=2$.
If, instead, $n_\pm=1$, then Lemma \ref{L:4.4}, upon choosing the fundamental matrix underlying the definition of quasi-boundary values appropriately, shows that the first component of $\cevec u$ is constant.
The requirement $\cevec u(b)=\cevec u(a)$ is then represented by the choice $A=(0,1,0,-1)$ which shows again that $T$ is self-adjoint.

Now we show that $T_K\subset T$.
To this end assume that $(u,f)\in T_K$.
Then there exists a sequence $(u_n,f_n)\in T_\mn$ such that $f_n\to f$ and $\langle u_n,f_n\rangle\to \langle u,f\rangle$.
For each $(u_n,f_n)\in T_\mn\subset T$ we have $0=\cevec{u}_n(b)-\cevec{u}_n(a)=\int_{(a,b)}U(\cdot,0)^*wf_n$.
Since $f_n\to f$ this gives
\[\cevec{u}(b)-\cevec{u}(a)=\int_{(a,b)}U(\cdot,0)^*wf=0\]
implying $T_K\subset T$ as promised.

If $\mc L_0$ is trivial, then $\ker T_\mx$ has dimension $2$, i.e., $R$ is a $2$-dimensional extension of $T_\mn$.
$T$ is also a $2$-dimensional extension of $T_\mn$ since $n_\pm=2$.
Thus $R=T$ and this completes the proof for this case.

Finally, if $\mc L_0$ is not trivial, then both $R$ and $T$ are $1$-dimensional extensions of $T_\mn$ and again identical.
\end{proof}

We have shown that, under the hypotheses of Theorem \ref{T:6.3}, the Krein-von Neumann extension can be described by periodic quasi-boundary values.
We show now that an alternative way is to describe it by a Dirichlet-to-Neumann-type relation.
This is particularly interesting in the case of regular endpoints on which we concentrate now.
Choose the fundamental matrix $U$ for $Ju'+qu=0$ by the requiring $U(b)=J$ (note that $U$ is then real-valued).
Then $(u,f)\in T_K$ if and only if $u(b)=\cevec u(b)=\cevec u(a)=U(a)^*Ju(a)$.
The ``Dirichlet'' data $D$ of $(u,f)$ are given by the first components of the boundary data, i.e., by $D=(u_1(a),u_1(b))=(u_1(a),U_{2,1}(a)u_1(a)-U_{1,1}(a)u_2(a))$
where $U_{j,k}$ denotes the element in row $j$ and column $k$ of $U$.
The ``Neumann'' data are given by $N=(-u_2(a),u_2(b))$.\footnote{We put the extra minus sign in the first entry of $N$ because Neumann data use outward derivatives.}
Thus $N=(-u_2(a),U_{2,2}(a)u_1(a)-U_{1,2}(a)u_2(a))$.
The ``Dirichlet-to-Neumann'' relation $M$ is therefore given as
$$M=\operatorname{span}\left\{\left(\bsm{1\\ U_{2,1}(a)},\bsm{0\\ U_{2,2}(a)}\right), \left(\bsm{0\\ U_{1,1}(a)},\bsm{1\\ U_{1,2}(a)}\right)\right\}.$$
If $U_{1,1}(a)=0$ we have that $0$ is a ``Dirichlet'' eigenvalue, i.e., an eigenvalue of the self-adjoint relation
$$T_D=\{(u,f)\in T_\mx:(1,0)u(a)=(1,0)u(b)=0\}.$$
Otherwise $M$ is the graph of a linear map with matrix representation
\[M=\frac{1}{U_{1,1}(a)}\begin{pmatrix} -U_{2,1}(a)&1\\ 1&U_{1,2}(a) \end{pmatrix}\]
using the fact that $J=U(x)^*JU(x)=\det U(x)J$ and hence $\det U(x)=1$.
We may now write $T_K=\{(u,f)\in T_\mx:(D,N)\in M\}$ or, when $0$ is not a Dirichlet eigenvalue, $T_K=\{(u,f)\in T_\mx:N=MD\}$.

In Theorem \ref{T:6.3} we assumed that $0\not\in\Lambda$.
If, instead, $0$ is in $\Lambda$, a simple example shows that the conclusion of Theorem \ref{T:6.3} does not hold.
The example is as follows.
Let $q=\sm{0&2\\2&0}\delta_0$ and $w=\sm{2&0\\0&2}\delta_0$.
Since $u\in\ker T_\mx$ if and only if $u(x)=\alpha\sm{1\\ 0}$ if $x<0$ and $u(x)=\beta\sm{0\\ 1}$ if $x>0$ one finds $T_\mn=\{(0,0)\}$ and
$$T_K = L^2(w)\times\{0\}=\{(u,f)\in T_\mx:(0,1)u(a)=0=(1,0)u(b)\},$$
i.e., $T_K$ is determined by separated boundary conditions.

In fact, this situation persists in greater generality as we will show next.
If both endpoints $a$ and $b$ are limit-circle and $\Xi_0$ is not empty but finite, let $\Xi_0=\{x_1,...,x_N\}$ where $x_1<...<x_N$.
Also set $x_0=a$ and $x_{N+1}=b$.
Using Lemma \ref{P:3.3} we know that any element of $\ker\mc T_\mx$ is a linear combination of functions supported on $\ov{(x_j,x_{j+1})}$.
Each such interval can support at most one linearly independent function $u$ in $\ker\mc T_\mx$ since such a solution must satisfy that $u^+(x_j)\in\ker B_+(x_j,0)$ for $j=1,..., N$ and $u^-(x_{j+1})\in\ker B_-(x_{j+1},0)$ for $j=0,...,N-1$.
The compactly supported functions in $\ker\mc T_\mx$ are, of course in $\ker\mc T_\mn$.
Now choose symplectic fundamental matrices $U_\ell$ and $U_r$ for $Ju'+qu=0$ on the intervals $(a,x_1)$ and $(x_N,b)$, respectively.
As in Section \ref{S:4} we define $u\mapsto \cev u$ where $\cev u(x)=U_\ell(x)^*Ju(x)$ and $u\mapsto \vec u$ where $\vec u(x)=U_r(x)^*Ju(x)$.
Also, let $v_0$ and $v_N$ be non-trivial solutions of $Ju'+qu=0$ on $(a,b)$ but supported only on $(a,x_1]$ and $[x_N,b)$, respectively.

\begin{thm}\label{T:6.4}
Suppose both endpoints $a$ and $b$ are limit-circle, $0\in\Lambda$, $\Xi_0$ is finite, and $T_\mn$ is non-negative.
Then the Krein-von Neumann extension is given by
$$T_K=\ov{T_\mn}\setsum(\ker T_\mx\times\{0\})=\{(u,f)\in T_\mx:\cev v_0(a)^*J\cev u(a)=0=\vec v_N(b)^*J\vec{u}(b)\}.$$
In particular, $T_K$ is self-adjoint.
\end{thm}

We emphasize that $\ov{T_\mn}\cap(\ker T_\mx\times\{0\})$ does not have to be trivial since there may be compactly supported elements of $\ker{\mc T_\mx}$ of positive norm.
We use the symbol $\setsum$ to denote the span of the union of two subspaces $A$ and $B$ of a given vector space by $A\setsum B$.

\begin{proof}
Assume first that $[v_0]$ and $[v_N]$ are linearly independent.
Note that $[v_0]$ and $[v_N]$ are in $\ker T_\mx$ but not in $\ker\ov{T_\mn}$.
Thus
$$R=\ov{T_\mn}\setsum(\ker T_\mx\times\{0\})$$
is a $2$-dimensional extension of $T_\mn$ and, as before, a subset of $T_K$.
Let
$$T=\{(u,f)\in T_\mx:\cev v_0(a)^*J\cev u(a)=0=\vec v_N(b)^*J\vec{u}(b)\}.$$
As in the proof of Theorem \ref{T:6.3} we prove now that $T_K\subset T$ which we know from Section \ref{S:4.1} to be a self-adjoint extension of $T_\mn$.
If $(u,f)\in T_K$, then there is a sequence $(u_n,f_n)\in T_\mn\subset T$ with $f_n\to f$ and $\<u_n,f_n\>\to\<u,f\>$.
Therefore $0=\vec v_N(b)^*J\vec u_n(b)=\int v_N^*wf_n$ which implies, upon taking the limit as $n$ tends to infinity, that $\vec v_N(b)^*J\vec u(b)=\int v_N^*wf=0$.
Similarly, $\cev v_0(a)^*J\cev u(a)=0$, i.e, $(u,f)\in T$.
Hence we have $R\subset T_K\subset T$ where both $R$ and $T$ are $2$-dimensional extensions of $T_\mn$.
Hence $R=T_K=T$.

Now suppose that $[v_0]$ and $[v_N]$ are linearly dependent but not both $0$, i.e., there are numbers $\alpha$ and $\beta$, not both $0$, such that $\alpha v_0+\beta v_N=\ell_0\in\mc L_0$.
In this case $R$ is a $1$-dimensional extension of $T_\mn$.
If neither $\alpha$ nor $\beta$ is $0$, one of the equations $\cevec v_0(a)^*J\cevec u(a)=0$ and $\cevec v_N(b)^*J\cevec{u}(b)=0$ implies the other, since $\cevec\ell_0(b)J\cevec u(b)-\cevec\ell_0(a)J\cevec u(a)=0$ for all $(u,f)\in\mc T_\mx$ and since $\ell_0=\alpha v_0$ near $a$ and $\ell_0=\beta v_N$ near $b$.
Thus $T$ is a $1$-dimensional extension of $T_\mn$.
Otherwise, if, say $\alpha=0$, the condition $\vec v_N(b)^*J\vec{u}(b)=0$ is vacuous and, again, $T$ is a $1$-dimensional extension of $T_\mn$.
Either way, $R=T_k=T$.

Finally, if $[v_0]=[v_N]=0$, then our results of Section \ref{S:3} show that $n_\pm=0$ in which case $T_K=R=\ov{T_\mn}$.
The conditions defining $T$ are now both vacuous, since $v_0$ and $v_N$ are in $\mc L_0$. Therefore $T=\ov{T_\mn}$, too.
\end{proof}

We close with the observation that for the Friedrichs extension there is no specific type of boundary condition even if we restrict ourselves to problems with two limit-circle endpoints and where $0$ is not in $\Lambda$.
Indeed, we have the following examples in each of which we choose $J=\sm{0&-1\\ 1&0}$ and $(a,b)=(-2,2)$.
Our conditions will guarantee that $0\not\in\Lambda$, $T_\mn$ is non-negative, and the endpoints are regular.
\begin{enumerate}
  \item When $q=0$ and $w=2\delta_0\id$ the Friedrichs extension is given by anti-periodic boundary conditions.
  \item When $q=0$ and $w=2(\delta_0+\delta_1)\id$ the Friedrichs extension is given by periodic boundary conditions.
  Incidentally, the Krein-von Neumann and the Friedrichs extension coincide in this case as it is the only non-negative extension of $T_\mn$.
  \item Let $q=\sm{v&0\\ 0&-1/p}$ and $w=\sm{r&0\\ 0&0}$ where $1/p$ and $r$ are positive and integrable while $v$ is non-negative and integrable with respect to Lebesgue measure.
  It is then well-known (see for example Theorem 6.8.5 (iii) in \cite{MR3971207} or Remark 4.3.11 in \cite{GNZ}) that the Friedrichs extension is given by Dirichlet conditions, i.e., $(1,0)u(-2)=(1,0)u(2)=0$.
\end{enumerate}

\appendix
\section{The Friedrichs extension}\label{A:A}
We give a short presentation of the Friedrichs extension for linear relations roughly following Weidmann's strategy for operators \cite{MR566954}.

Let $\mc H$ be a Hilbert space with scalar product $\scp{\cdot,\cdot}$ and associated norm $\norm{\cdot}$.
Also let $S$ be a symmetric linear relation in $\mc H\times \mc H$ and denote its domain by $\mc V$.
Assume that $S$ is bounded below by $\alpha\in\bb R$, i.e., that $\scp{f,u}\geq\alpha\norm{u}^2$ for all $(u,f)\in S$.
If $(u,f)$, $(u,f')$, and $(v,g)$ are in $S\subset S^*$, then $(0,f-f')\in S^*$ and therefore $\scp{f-f',v}=\scp{0,g}=0$.
Thus the map $\mc V\times \mc V\to\bb C: (u,v)\mapsto \scp{f,v}=\scp{u,g}$ is well-defined.
Since $S$ is bounded below by $\alpha$
\begin{equation}\label{dH}
\dscp{u,v}=\scp{f,v}+(1-\alpha)\scp{u,v}
\end{equation}
is a scalar product in $\mc V$.
We denote the norm associated with $\dscp{\cdot,\cdot}$ by $\tnorm\cdot$.
The completion of $\mc V$ under $\tnorm\cdot$ will be called $\mc W$.
For $u$ and $v$ in $\mc W$ choose sequences $u_n\in\mc V$ and $v_n\in \mc V$ such that $\tnorm{u_n-u}$ and $\tnorm{v_n-v}$ tend to $0$.
Since $\dscp{u,v}=\lim_{n\to\infty}\dscp{u_n,v_n}$ does not depend on how these sequences are chosen, we may extend the inner product $\dscp{\cdot,\cdot}$ to $\mc W$ which, thereby, becomes a Hilbert space.
We also introduce $\mc X$ as the closure of $\mc V$ in $\mc H$.

\begin{lem}\label{L:A1}
$\tnorm\cdot$ is $\norm\cdot$-compatible, that is, the following two properties hold.
(i) $\norm v\leq\tnorm v$ for all $v\in\mc V$ and (ii) whenever $n\mapsto v_n\in\mc V$ is a Cauchy sequence with respect to $\tnorm{\cdot}$ and $\norm{v_n}\to0$, then also $\tnorm{v_n}\to0$.
\end{lem}

\begin{proof}
That (i) holds is obvious.
To show (ii) let $n\mapsto v_n\in \mc V$ be a Cauchy sequence with respect to $\tnorm{\cdot}$ for which $\norm{v_n}\to0$.
Let $v$ be the limit of $v_n$ with respect to $\tnorm{\cdot}$ and let $(u,f)$ be an arbitrary element of $S$.
Then $\dscp{u,v_n}=\scp{f,v_n}+(1-\alpha)\scp{u,v_n}$ tends to $0$.
This implies that $\dscp{u,v}=0$ for all $u\in \mc V$ and hence $v=0$.
\end{proof}

\begin{lem}\label{L:A2}
$\mc W$ may be continuously embedded{\,}\footnote{Recall that an embedding is a linear injection.} into $\mc X$.
\end{lem}

\begin{proof}
For $f$ in $\mc W$ choose a sequence $n\mapsto f_n\in \mc V$ such that $\tnorm{f_n-f}\to0$.
This sequence is Cauchy with respect to $\tnorm{\cdot}$.
Therefore it is Cauchy and thus convergent in $\mc X$ with respect to $\norm{\cdot}$.
Denote the corresponding limit by $\phi(f)$, observing that $\phi(f)$ remains unchanged if a different sequence in $\mc V$ approaching $f$ is chosen.
Thus we have defined a map $\phi:\mc W\to\mc X$.
It is clear that $\phi$ is linear.
To show it is injective assume $\phi(f)=0$.
Suppose $\tnorm{f_n-f}\to0$ with $f_n\in\mc V$.
Then $n\mapsto f_n$ is Cauchy with respect to $\tnorm{\cdot}$ and $\norm{f_n}=\norm{f_n-\phi(f)}\to0$.
By Lemma \ref{L:A1} we have $\tnorm{f_n}\to0$.
Hence
$$\tnorm{f}\leq\tnorm{f-f_n}+\tnorm{f_n}\to0$$
so that $f=0$ proving injectivity of $\phi$.

Finally, $\phi$ is bounded by $1$ (and hence continuous) since
\[\norm{\phi(f)}=\lim_{n\to\infty} \norm{f_n}\leq \lim_{n\to\infty} \tnorm{f_n}=\tnorm{f},\]
if $\tnorm{f_n-f}\to0$ with $f_n\in\mc V$.
\end{proof}

Identifying $\phi(f)$ and $f$ we may consider $\mc W$ to be subset of $\mc X$ and hence of $\mc H$.

We can now define the Friedrichs extension $S_F$ of $S$.

\begin{thm}\label{Friedrichs}
Suppose $S$ is a symmetric linear relation in $\mc H\times\mc H$ which is bounded below.
Then the relation
$$S_F=\{(u,f)\in S^*: \exists\; n\mapsto (u_n,f_n)\in S: u_n\to u \text{ and } \scp{f_n,u_n}\to\scp{f,u}\}$$
is a self-adjoint extension of $S$.
\end{thm}

\begin{proof}
We freely use the notation established above, in particular we let $\alpha$ be a lower bound of $S$.

Fix $v\in\mc H$ and consider the linear functional $L:\mc W\to\bb C:u\mapsto \scp{v,u}$.
Since $|\scp{v,u}|\leq \norm{v}\norm{u}\leq \norm{v}\,\tnorm{u}$, $L$ is bounded and we have that $\scp{v,u}=\dscp{Rv,u}$ for some uniquely determined vector $Rv\in\mc W$.
Therefore we can define a linear map $R:\mc H\to\mc H$ with range in $\mc W$.
Since $0\leq\tnorm{Rv}^2=\scp{v,Rv}\leq \norm{v}\,\tnorm{Rv}$ it follows that $R$ is bounded below by $0$ and that it is continuous.

We now show that $R$ is self-adjoint.
Since
$$\scp{Rv,u}=\ov{\scp{u,Rv}}=\ov{\dscp{Ru,Rv}}=\dscp{Rv,Ru}=\scp{v,Ru}$$
when $(u,Ru)$ and $(v,Rv)$ are in $R$, we have that $R\subset R^*$.
To prove that $R^*\subset R$ let $(f,g)\in R^*$ and $(u,Ru)\in R$.
Then
$$\scp{g,u}=\scp{f,Ru}=\dscp{Rf,Ru}=\ov{\dscp{Ru,Rf}}=\ov{\scp{u,Rf}}=\scp{Rf,u}.$$
Hence $\scp{Rf-g,u}=0$ for all $u\in\mc H$ so that $g=Rf$ and hence $(f,g)=(f,Rf)\in R$.

Now define the self-adjoint linear relation $T=R^{-1}+(\alpha-1)\id\subset\mc H\times\mc H$.
Note that $\dom T=\ran R$ is dense in $(\ker R)^\perp=\mc W$.
To show that $T$ is an extension of $S$ consider an element $(u,f)\in S$ and an element $v\in\mc W$.
Then
$$\dscp{Rf,v}=\scp{f,v}=\dscp{u,v}+(\alpha-1)\scp{u,v}=\dscp{u+(\alpha-1)Ru,v}.$$
It follows that $u+(\alpha-1)Ru=Rf$ or, equivalently, $(u,f)\in R^{-1}+(\alpha-1)\id=T$, i.e., $S\subset T$.
Since $T$ is self-adjoint, it follows that $T\subset S^*$.

If $(u,f)\in T$ and hence $u\in\mc W$, there is a sequence $(u_n,f_n)\in S$ such that $\lim_{n\to\infty}\tnorm{u_n-u}=0$.
Since $\tnorm{\cdot}$ is $\norm{\cdot}$-compatible we get $\lim_{n\to\infty}\norm{u_n-u}=0$.
Moreover,
$$\scp{f_n,u_n}-\scp{f,u} = \tnorm{u_n}^2-\tnorm{u}^2+(\alpha-1)(\norm{u_n}^2-\norm{u}^2).$$
The triangle inequality shows now that the right-hand side in the previous identity tends to $0$.
We have therefore shown that $T\subset S_F$, when $S_F$ is given as stated in the theorem.

To prove the converse conclusion note first that $S_F$ is symmetric since $\scp{f,u}\in\bb R$ for all $(u,f)\in S_F$. 
Hence we have now $T\subset S_F\subset S_F^*$ and this implies 
$$S_F^*\subset T^*=T\subset S_F\subset S_F^*$$
forcing all inclusions to be equalities and concluding our proof. 
\end{proof}

\bibliographystyle{plain}

\end{document}